\documentclass[11pt,a4paper]{article}
\usepackage{xspace}
\usepackage{mathrsfs}
\usepackage{amsmath,amssymb,amsthm} 
\usepackage{pstricks,pst-node,pst-tree}

\newcommand{\resp}{\mbox{resp.}\xspace}
\newcommand{\ie}{\textit{i.e.}\xspace}

\newcommand{\Z}{\ensuremath{\mathbb{Z}}\xspace}

\newcommand{\Sig}{\mathfrak{S}}
\newcommand{\im}{\operatorname{\mathrm{Im}}}

\newcommand{\B}{\ensuremath{\mathbb{F}_2}\xspace}
\newcommand{\J}{\operatorname{\mathscr{J}\!}}
\newcommand{\G}{\mathscr{G}}
\newcommand{\intervalle}[1]{{[#1[}}
\newcommand{\inter}{\intervalle{i,j}}
\newcommand{\zinter}{\intervalle{0,i}}
\newcommand{\interz}{\intervalle{i,0}}

\newtheorem{lemma}{Lemma}
\newtheorem{theorem}{Theorem}

\newtheorem{proposition}{Proposition}
\newtheorem{questions}{Questions}

\newtheorem{remark}{Remark}
\newtheorem{definition}{Definition}
\newtheorem{thma}{Theorem A}
\newtheorem{thmaa}{Theorem A'}
\newtheorem{thmb}{Theorem B}

\begin{document}
\title{Negative local feedbacks in Boolean networks}
\author{Paul Ruet\\[2mm]
CNRS \\
IRIF, Universit\'{e} Paris Diderot - Paris 7 \\
Case 7014, 75205 Paris Cedex 13, France \\[2mm]
{\tt ruet@irif.fr}}
\maketitle
\begin{abstract}
We study the asymptotic dynamical properties of Boolean networks without local negative cycle. While the properties of Boolean networks without local cycle or without local positive cycle are rather well understood, recent literature raises the following two questions about networks without local negative cycle. Do they have at least one fixed point? Should all their attractors be fixed points? The two main results of this paper are negative answers to both questions: we show that and-nets without local negative cycle may have no fixed point, and that Boolean networks without local negative cycle may have antipodal attractive cycles.
\end{abstract}

\section{Introduction}
A Boolean network is a map $f$ from $\B^n$ to itself, where $n$ is a positive integer and $\B$ is the two-element field. We view $f$ as representing the dynamics of $n$ interacting components which can take two values, $0$ and $1$: at a state $x\in\B^n$, the coordinates which can be updated are the integers $i\in\{1,\ldots,n\}$ such that $f_i(x)\neq x_i$. Several dynamical systems can therefore be associated to $f$, depending on the choice of update scheme. In the synchronous dynamics \cite{Kau93,DES08}, all coordinates are updated simultaneously (it is simply the iteration of $f$), while in the (nondeterministic) asynchronous dynamics \cite{Tho91}, one coordinate is updated at a time, if any. Other dynamics are considered in the literature (e.g. random \cite{Kau69}), as well as comparisons between update schemes \cite{GS08}.

Boolean networks have plenty of applications. In particular, they have been extensively used as discrete models of various biological networks, since the early works of McCulloch and Pitts \cite{MCP43}, S. Kauffman \cite{Kau69} and R. Thomas \cite{Tho73}.

To a Boolean network $f$, it is possible to associate, for each state $x$, a directed graph $\G(f)(x)$ representing local influences between components $1,\ldots,n$ and defined in a way similar to Jacobian matrices for differentiable maps. Local feedbacks, \ie cycles in these local interaction graphs $\G(f)(x)$, have an impact on fixed points of $f$: \cite{SD05} proves that Boolean networks without local cycle have a unique fixed point.

On the other hand, the edges of $\G(f)(x)$ naturally come up with a sign, which is positive in case of a covariant influence and negative otherwise. Intuitively, when applied to the modeling of, e.g., gene regulatory networks, positive and negative signs correspond respectively to activatory and inhibitory effects. It is therefore expected that the dynamics associated with positive and negative cycles (the sign being the product of the signs of the edges) will in general be very different, and the biologist R. Thomas \cite{Tho81,TK01} proposed rules relating positive cycles to multistationarity (which corresponds to cellular differentiation in the field of gene networks) and negative cycles to sustained oscillations (a form of homeostasis).

In terms of Boolean networks, sustained oscillations can be interpreted either by an attractive cycle (a cycle in the asynchronous dynamics which cannot be escaped), or more generaly by a cyclic attractor (a strongly connected component of the asynchronous dynamics which does not consist in a fixed point). Also notice that the absence of fixed point entails a cyclic attractor. Therefore, these rules give rise to the following series of questions:
\begin{questions}
\begin{enumerate}
\item Does a network without local positive cycle have at most one fixed point?
\item Does a network without local negative cycle have at least one fixed point?
\item Does a network with a cyclic attractor have a local negative cycle?
\item Does a network with an attractive cycle have a local negative cycle?
\end{enumerate}
\end{questions}
Question 1 is given a positive answer in \cite{RRT08}.

Question 2, which is mentioned for instance in \cite{Ric10}, is also a negative counterpart of Question 1, and thus motivated by the above result of \cite{SD05}. Question 3 is related to the following result: if $f$ has a cyclic attractor, then the global interaction graph $\G(f)$ obtained by taking the union of the local graphs $\G(f)(x)$ has a negative cycle \cite{Ric10}. Several partial results are also known for local negative cycles \cite{Ric11,RR12}, and are recalled in Section \ref{sec:known}. In the more general discrete case (with more than two values), \cite{Ric10} shows that a network without local negative cycle may have an attractive cycle and no fixed point. A partial positive answer to Question 4 is known for Boolean networks of a special class called and-nets (in which all dependencies are conjunctions): and-nets with a special type of attractive cycle, called antipodal, do have a local negative cycle \cite{Rue14}.

Theorem A gives a negative answer to Question 2, and hence to Question 3, even for and-nets. In Section \ref{sec:andfp}, we construct a $12$-dimensional and-net with no local negative cycle and no fixed point. The proof relies essentially on a trick for delocalizing cycles by expanding and-nets (Section \ref{sec:qdf}). Section \ref{sec:ker} also mentions a consequence for kernels in graph theory (Theorem A').

Then Theorem B gives a negative answer to Question 4: in Section \ref{sec:attcyc}, we prove that arbitrary Boolean networks without local negative cycle may have (antipodal) attractive cycles. For this construction, we start with a Boolean network with an antipodal attractive cycle, and then modify the neighborhood of this attractive cycle so as to delocalize all negative cycles. The proof that the resulting network has indeed no local negative cycle is simplified by using some isometries of $\B^n$ (Sections \ref{sec:def} and \ref{sec:equivariance}). We may remark that the metric structure of $\B^n$ was the main ingredient for unsigned cycles and positive cycles too (see \cite{Rue14}), though the proofs were apparently very different.

Section \ref{sec:addition} includes remarks on non-expansive Boolean networks, hoopings, invertible Jacobian matrices, and reduction of networks.

\section{Definitions and statement of results}
\label{sec:bool}
Let $\{e^1,\ldots,e^n\}$ denote the canonical basis of the vector space $\B^n$, and for each subset $I$ of $\{1,\ldots,n\}$, let $e^I=\sum_{i\in I} e^i$, where the sum is the sum of the field \B. We may remove some brackets and write $e^{1,2}$ for $e^{\{1,2\}}$ for instance. For $x,y\in\B^n$, $d(x,y)$ denotes the Hamming distance, \ie the cardinality of the unique subset $I\subseteq\{1,\ldots,n\}$ such that $x+y=e^I$.

\subsection{Boolean networks}
A Boolean network is a map $f:\B^n\rightarrow\B^n$. To such a map, it is possible to associate several dynamics with points of $\B^n$ as the states.

The \emph{synchronous dynamics} is simply the iteration of $f$. The \emph{asynchronous dynamics} is the directed graph $\Gamma(f)$ with vertex set $\B^n$ and an edge from $x$ to $y$ when for some $i$, $y=x+e^i$ and $f_i(x)\neq x_i$. It is a nondeterministic dynamics (a state $x\in\B^n$ can have $0$ or several successors) in which at most one coordinate is updated at a time. The coordinates $i$ such that $f_i(x)\neq x_i$ are those which can be updated in state $x$, and may therefore naturally be viewed as the \emph{degrees of freedom} of $x$.

The asynchronous dynamics, illustrated in Figure \ref{fig:bn}, can be viewed as a weak form of orientation of the Boolean hypercube $\B^n$, in which each undirected edge is replaced by $0,1$ or $2$ of the possible choices of orientation.

It is easily seen that $f$ can be recovered from $\Gamma(f)$: $f(x)=x+e^I$, where $\{(x,x+e^i), i\in I\}$ is the set of edges leaving $x$ in $\Gamma(f)$.
\begin{figure}[!t]
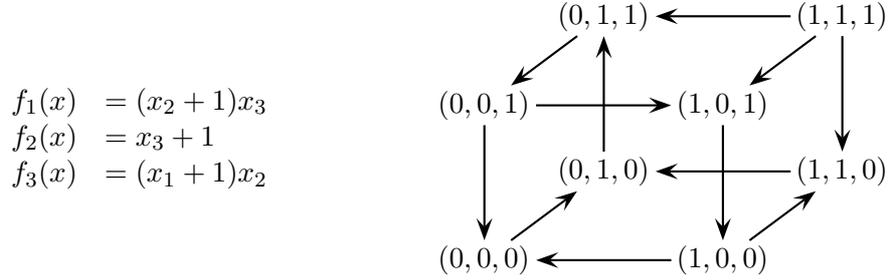

$$
\begin{array}{cc}
\begin{array}{rl}
f_1(x) &= (x_2+1)x_3 \\
f_2(x) &= x_3+1 \\
f_3(x) &= (x_1+1)x_2
\end{array}
&
\qquad\qquad
\begin{array}{cccc}
& \rnode{011}{(0,1,1)} && \rnode{111}{(1,1,1)} \\[7mm]
\rnode{001}{(0,0,1)} && \rnode{101}{(1,0,1)} & \\[4mm]
& \rnode{010}{(0,1,0)} && \rnode{110}{(1,1,0)} \\[7mm]
\rnode{000}{(0,0,0)} && \rnode{100}{(1,0,0)} &
\psset{nodesep=2pt,arrowsize=6pt}
\ncline[arrows=<-]{000}{100}
\ncline[arrows=->]{000}{010}
\ncline[arrows=<-]{000}{001}
\ncline[arrows=->]{111}{011}
\ncline[arrows=->]{111}{101}
\ncline[arrows=->]{111}{110}
\ncline[arrows=->]{100}{110}
\ncline[arrows=->]{110}{010}
\ncline[arrows=->]{010}{011}
\ncline[arrows=->]{011}{001}
\ncline[arrows=->]{001}{101}
\ncline[arrows=->]{101}{100}
\end{array}
\end{array}
$$
\caption{A map $f:\B^3\rightarrow\B^3$ and the asynchronous dynamics $\Gamma(f)$ associated to it.}
\label{fig:bn}
\end{figure}

We shall be essentially interested in asymptotic dynamical properties. Both dynamics agree on fixed points. On the other hand, a \emph{trajectory} will be a path in the asynchronous dynamics $\Gamma(f)$, and an \emph{attractor} a terminal strongly connected component of $\Gamma(f)$. An attractor which is not a singleton (\ie which does not consist in a fixed point) is called a \emph{cyclic attractor}. In particular, a network with no fixed point must have at least one cyclic attractor. In the case of the fixed-point-free network $f$ of Figure \ref{fig:bn}, the unique cyclic attractor consists in the subgraph of $\Gamma(f)$ induced by $\B^3\setminus\{(1,1,1)\}$.

\emph{Attractive cycles}, \ie, cyclic trajectories $\theta$ such that for each point $x\in\theta$, $d(x,f(x))=1$, are examples of cyclic attractors. Observe that attractive cycles are deterministic, since any point in $\theta$ has a unique degree of freedom, hence they can also be defined as cycles of the synchronous dynamics in which exactly one coordinate is updated at a time.

The \emph{antipode} of $x\in\B^n$ is $\overline{x}=x+e^{1,\ldots,n}$. \emph{Antipodal} attractive cycles are those obtained from the attractive cycle
\begin{align*}
& (0,e^1,e^{1,2},\ldots,e^{1,\ldots,n-1},e^{1,\ldots,n},e^{2,\ldots,n},,e^n,0) \\
=\ & (0,e^1,e^{1,2},\ldots,e^{1,\ldots,n-1},\overline{0},\overline{e^1},\overline{e^{1,2}},\ldots,\overline{e^{1,\ldots,n-1}},0)
\end{align*}
by translations and permutations of coordinates. An antipodal attractive cycle is therefore of the form
$$
(x^0,x^1,\ldots,x^n,\overline{x^0},\overline{x^1},\ldots,\overline{x^n},x^0)
$$
where $x^i=x^{i-1}+e^{\sigma(i)}=f(x^{i-1})$ for each $i\in\{1,\ldots,n\}$, and $\sigma$ is some permutation of $\{1,\ldots,n\}$.

\subsection{Interaction graphs}
\label{sec:int}
As the network terminology suggests, a Boolean network $f:\B^n\rightarrow\B^n$ induces directed graphs which represent interactions between its variables $x_1,\ldots,x_n$.

Given $\varphi:\B^n\rightarrow\B$ and $i\in\{1,\ldots,n\}$, the \emph{discrete $i^{th}$ partial derivative} $\partial\varphi/\partial x_i=\partial_i\varphi:\B^n\rightarrow\B$ maps each $x\in\B^n$ to
$$
\partial_i\varphi(x)=\varphi(x)+\varphi(x+e^i),
$$
where the $+$ here is again the addition of the field \B, so that $\partial_i\varphi(x)=1$ if and only if $\varphi(x)\neq\varphi(x+e^i)$. In that case, the influence of variable $x_i$ on $\varphi$ at $x$ is either covariant when the map
$$
\B\rightarrow\B, \ \alpha\mapsto \varphi(x_1,\ldots,x_{i-1},\alpha,x_{i+1},\ldots,x_n)
$$
is increasing, or contravariant when it is decreasing. Given $f:\B^n\rightarrow\B^n$ and $x\in\B^n$, the \emph{discrete Jacobian matrix} $\J(f)(x)$ is the $n\times n$ matrix with entries $\J(f)(x)_{i,j}=\partial_jf_i(x)$.

A \emph{signed directed graph} is a directed graph with a sign, $+1$ or $-1$, attached to each edge, and the \emph{sign} of a cycle (or more generally the sign of a path) is defined to be the product of the signs of its edges.

All cycles considered in the paper will be elementary.

The \emph{interaction graph of $f$ at $x$}, is defined \cite{RRT08} to be the signed directed graph $\G(f)(x)$ on vertex set $\{1,\ldots,n\}$ which has an edge from $j$ to $i$ when $\J(f)(x)_{i,j}=1$, with positive (\resp negative) sign when the influence of $x_j$ on $f_i$ is covariant (\resp contravariant). It is straightforward to verify that the condition for an edge to be positive is equivalent to:
$$
x_j=f_i(x).
$$
The \emph{global interaction graph} $\G(f)$ has the same vertices, and a positive (resp. negative) edge from $j$ to $i$ when for some $x$, $\G(f)(x)$ has. A consequence of this definition is that a global interaction graph may have two edges of opposite signs from some vertex to another. A cycle, or more generally a path, of $\G(f)$ is said to be \emph{local} when it lies in $\G(f)(x)$ for some $x$.

\subsection{And-nets}
In Section \ref{sec:andfp}, we shall be interested in a class of Boolean networks called and-nets. A map $f:\B^n\rightarrow\B^n$ is called an \emph{and-net} when for each $i\in\{1,\ldots,n\}$, $f_i$ is a product of literals, \ie there exist disjoint subsets $P_i$ and $N_i$ of $\{1,\ldots,n\}$ such that
$$
f_i(x)=\prod_{j\in P_i}x_j\prod_{j\in N_i}(x_j+1),
$$
with the convention that the empty product is $1$. Indices in $P_i$ (\resp in $N_i$) are called the positive (\resp negative) \emph{inputs} of $f_i$: they are indeed the vertices $j$ of $\G(f)$ such that $(j,i)$ is a positive (\resp negative) edge of $\G(f)$.

The network of Figure \ref{fig:bn} is an example of and-net. As explained in \cite{RR12}, in the case of and-nets, the global interaction graph $\G(f)$ actually determines $f$ (a statement which obviously does not hold for arbitrary Boolean networks): given a signed directed graph $G$ which is \emph{simple} (\ie without parallel edges), define the and-net $f$ by
$$
f_i(x)=\prod_{(j,i)\in E^+(G)}x_j\prod_{(j,i)\in E^-(G)}(x_j+1),
$$
where $E^+(G)$ (\resp $E^-(G)$) denotes the set of positive (\resp negative) edges of $G$. Then $f$ is the unique and-net such that $\G(f)=G$.

Proposition \ref{prop:andlocal}, which is proved in \cite{RR12}, shows that, for an and-net $f$, locality of a cycle $C$ of $\G(f)$ can be expressed as the absence of certain specific subgraphs of $\G(f)$ called delocalizing triples, the definition of which we recall now. Given a simple signed directed graph $G$ and a cycle $C$ of $G$, a triple $(i,j,k)\in\{1,\ldots,n\}^3$ is said to be a \emph{delocalizing triple} of $C$ when $j,k$ are distinct vertices of $C$ and $(i,j),(i,k)$ are two edges of $G$ that are
\begin{itemize}
\item not edges of $C$,
\item and of different signs.
\end{itemize}
A delocalizing triple $(i,j,k)$ of $C$ is said to be \emph{internal} when $i$ is a vertex of $C$, \emph{external} otherwise. See Figure \ref{fig:deloc}.
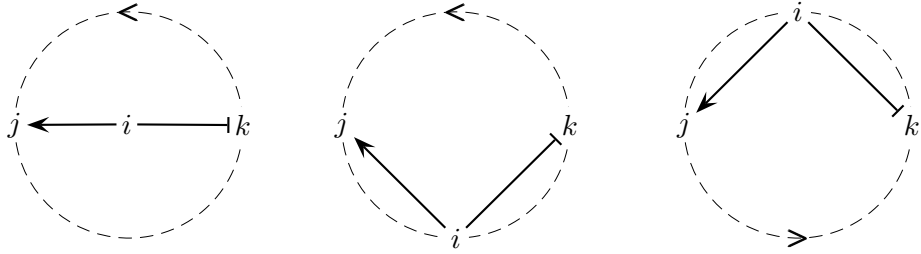
\begin{figure}[!t]
\centering
\begin{pspicture}(12, 3)
\psset{nodesep=2pt,arrowsize=6pt,labelsep=3pt}
\put(0,1.5){\rnode{V1}{$j$}}
\nccircle[angle=-90,linestyle=dashed,linewidth=0.3pt,nodesep=4pt]{-}{V1}{1.5}\ncput*{\rnode{V2}{$k$}}\ncput[npos=.75]{$\pmb{<}$}
\put(1.5,1.47){\rnode{U}{$i$}}
\ncline{->}{U}{V1}
\ncline{-|}{U}{V2}
\put(4.3,1.5){\rnode{VV1}{$j$}}
\nccircle[angle=-90,linestyle=dashed,linewidth=0.3pt,nodesep=4pt]{-}{VV1}{1.5}\ncput*{\rnode{VV2}{$k$}}\ncput*[npos=.25]{\rnode{UU}{$i$}}\ncput[npos=.75]{$\pmb{<}$}
\ncline{->}{UU}{VV1}
\ncline{-|}{UU}{VV2}
\put(8.8,1.5){\rnode{VVV1}{$j$}}
\nccircle[angle=-90,linestyle=dashed,linewidth=0.3pt,nodesep=4pt]{-}{VVV1}{1.5}\ncput*{\rnode{VVV2}{$k$}}\ncput*[npos=.75]{\rnode{UUU}{$i$}}\ncput[npos=.25]{$\pmb{>}$}
\ncline{->}{UUU}{VVV1}
\ncline{-|}{UUU}{VVV2}
\end{pspicture}
\caption{External or internal delocalizing triple of a cycle in a signed directed graph. Usual arrows denote positive edges, while arrows ending with a $\dashv$ denote negative edges.}
\label{fig:deloc}
\end{figure}
\begin{proposition}
\label{prop:andlocal}
Let $f:\B^n\rightarrow\B^n$ be an and-net. Given a cycle $C$ if $\G(f)$, $C$ is local if and only if it has no delocalizing triple \cite{RR12}.
\end{proposition}

\subsection{Statement of results}
\label{sec:known}
Let $f:\B^n\rightarrow\B^n$ be a Boolean network. Shih and Dong have proved in \cite{SD05} that if $\G(f)$ has no local cycle, then $f$ has a unique fixed point. But the sign of cycles has an influence of the dynamics too. For instance, \cite{RMCT03} shows that when the interaction graph $\G(f)(x)$ is independent of $x$ and consists in a positive (\resp negative) cycle with no other edge, $f$ has $2$ fixed points and no cyclic attractor (\resp $f$ has no fixed point and a unique attractive cycle). So, in this somehow elementary case, the dynamics associated with positive and negative cycles are very different.

Based on these results, \cite{RRT08} proved that, for an arbitrary network $f$, if $\G(f)$ has no local positive cycle, then $f$ has at most one fixed point. Moreover, \cite{RRT08} proved that if $f$ has an attractive cycle, then $\G(f)$ has a (global) negative cycle. This motivated interest in investigating dynamical properties related to local negative cycles, suggesting in particular that the absence of a local negative cycle might imply the existence of a fixed point, or some related property.

The following theorem reviews the known partial results on negative cycles.
\begin{theorem}
\label{th:knownneg}
Let $f:\B^n\rightarrow\B^n$ be a Boolean network.
\begin{enumerate}
\item If $f$ has an attractive cycle, then $\G(f)$ has a negative cycle \cite{RRT08}.
\item If $f$ has a cyclic attractor (in particular if $f$ has no fixed point), then $\G(f)$ has a negative cycle \cite{Ric10}.
\item If $f$ is non-expansive ($d(f(x),f(y))\leqslant d(x,y)$ for all $x,y$) and has no fixed point, then $\G(f)$ has a local negative cycle \cite{Ric11}.
\item If $f$ is an and-net and has no fixed point, then $\G(f)$ has an internally local negative cycle (a negative cycle without internal delocalizing triple) \cite{RR12}.
\item If $f$ is an and-net and has an antipodal attractive cycle, then $\G(f)$ has a local negative cycle \cite{Rue14}.
\end{enumerate}
\end{theorem}
On the other hand, as mentioned in the Introduction, in the more general discrete case (for maps from a finite set $\{0,\ldots,d\}^n$ to itself, with analogous definitions of local interaction graphs), \cite{Ric10} shows that, even for $d=3$ and $n=2$, there exists a network with no local negative cycle, no fixed point and an attractive cycle.

The main results of this paper are the following two theorems proved in Sections \ref{sec:andfp} and \ref{sec:attcyc}.

\begin{thma}
There exist and-nets with no local negative cycle and no fixed point (hence with a cyclic attractor).
\end{thma}

\begin{thmb}
There exist Boolean networks with no local negative cycle and an attractive cycle.
\end{thmb}

\section{First construction: And-nets without fixed point}
\label{sec:andfp}

This section is devoted to the proof of Theorem A.

\subsection{Chords}
Let us start with a remark on point $5$ of Theorem \ref{th:knownneg}.

If an and-net $f$ has an antipodal attractive cycle $\theta$, we may assume that $\theta$ is $(0,\ldots,e^{1,\ldots,n-1},\overline{0},\ldots,\overline{e^{1,\ldots,n-1}},0)$ up to translation and a permutation of coordinates, so that $\G(f)$ has a negative cycle $C=(1,2,\ldots,n,1)$: \cite{Rue14} proves that this cycle is local, but it is easy to observe that it is actually chordless. Indeed, if $C$ has a negative chord $(i,j)$, then $f_j(x)=0$ as soon as $x_i=1$.
\begin{itemize}
\item Now, if $j\leqslant i$, then $e^{1,\ldots,i}_j=1$. Since $e^{1,\ldots,i}_i=1$ as well, by the above remark, $f_j(e^{1,\ldots,i})=0$, hence $i+1$ and $j$ are degrees of freedom of $e^{1,\ldots,i}$. Since $(i,j)$ is a chord, $j\neq i+1\text{ mod } n$, and $e^{1,\ldots,i}$ has at least two degrees of freedom.
\item Otherwise $i\leqslant j-1$, so $e^{1,\ldots,j-1}_i=1$, and $f_j(e^{1,\ldots,j-1})=0$ by the above remark. Since $e^{1,\ldots,j-1}_j=0$ too, $j$ is not a degree of freedom of $e^{1,\ldots,j-1}$.
\end{itemize}
In both cases, we have a contradiction with the hypothesis that $\theta$ is an atractive cycle, and a similar argument applies for a positive chord at $\overline{e^{1,\ldots,i}}$.

Conversely, it is clear that if $\G(f)$ is a Hamiltonian negative cycle, then $f$ is an and-net and has an antipodal attractive cycle. We thus have the following:
\begin{remark}
$f$ is an and-net with an antipodal attractive cycle if and only if $\G(f)$ is a (chordless) Hamiltonian negative cycle.
\end{remark}
Since chordless Hamiltonian cycles are local \cite{RR08}, this entails point $5$ of Theorem \ref{th:knownneg}.

Now, remember that a \emph{kernel} of a directed graph is an independent and absorbent set of vertices \cite{vnM44,Ber70}, and that fixed points of a \emph{negative} and-net (all edges negative) are in one-to-one correspondence with kernels of the transpose of the underlying directed graph \cite{RR12}. It is well-known that there exist graphs $G$ without kernel such that every odd cycle of $G$ has as many chords as desired: \cite{GS82} defines, for every $k\geqslant 2$, a graph without kernel whose odd cycles all have at least $k$ chords. The graph we shall use as a seed for constructing the counter-example below is actually simpler than the ones defined in \cite{GS82}: every odd cycle of our graph has a single chord (Figure \ref{fig:ce}).

These graphs correspond to negative and-nets without fixed point, whose negative cycles all have chords. Is it true that if all negative cycles have a chord, they are non local? Of course not in general, but we show below that in the case of negative and-nets, if these chords are suitably distributed (existence of a quasi-delocalizing function, see Definition \ref{def:qdf}), the network can be deformed (through expansion to be defined below) so as to delocalize all negative cycles. We first recall the definitions of expansion and reduction (Section \ref{sec:red}), and then explain the delocalization process (Section \ref{sec:qdf}). Since, as we shall see, expansion preserves fixed points, we shall end up with and-nets with no local negative cycle and no fixed point, as claimed by Theorem A.

\subsection{Reduction and expansion}
\label{sec:red}
We adapt the definition of \cite{NRTC11} to our notation.

If $f:\B^n\rightarrow\B^n$ is a Boolean network whose global interaction graph $\G(f)$ has no \emph{loop} on $n$ (no edge $(n,n)$), it is possible to define a \emph{reduced} Boolean network $f':\B^{n-1}\rightarrow\B^{n-1}$ by substitution:
$$
f'(x) = f(x,f_n(x,0)) = f(x,f_n(x,1))
$$
for each $x\in\B^{n-1}$, because the hypothesis on $\G(f)$ entails $f_n(x,0)+f_n(x,1)=\partial_nf(x,0)=\partial_nf(x,1)=0$. We shall therefore write $f'(x)=f(x,f_n(x,-))$. If $f$ reduces to $f'$, we shall also say that $f$ is \emph{expanded} from $f'$. For any $x\in\B^{n-1}$, let
$$
x'=(x,f_n(x,-))\in\B^n,
$$
so that $f'(x)=f(x')$. If $\pi:\B^n\rightarrow\B^{n-1}$ is the projection, then clearly, $\pi(x')=x$.

In the above definition, for simplicity, we have considered only reductions obtained by substituting variable $x_n$, but reductions over any variable $x_i$ is obviously possible and will be considered later in the paper.

The following is proved in \cite{NRTC11}.
\begin{proposition}
\label{prop:red}
Let $f:\B^n\rightarrow\B^n$ be a Boolean network whose global interaction graph $\G(f)$ has no loop on $n$.
\begin{enumerate}
\item Fixed points are preserved by reduction and expansion: $x$ is a fixed point of $f'$ if and only if $x'$ is a fixed point of $f$.
\item Attractive cycles are preserved by reduction: $\pi$ maps attractive cycles of $f$ to attractive cycles of $f'$.
\end{enumerate}
\end{proposition}
Notice that attractive cycles are not preserved by expansion: for instance, $f:\B^3\rightarrow\B^3$ defined by $f(x_1,x_2,x_3)=(x_2+1,x_1,x_1+x_2)$ has no attractive cycle, but reduces to $(x_1,x_2)\mapsto(x_2+1,x_1)$ which clearly has one. It is also observed in \cite{NRTC11} that arbitrary cyclic attractors are not generally preserved by reduction. It is not difficult to show that they are not preserved by expansion either.

\subsection{Quasi-delocalizing functions}
\label{sec:qdf}
\begin{definition}
\label{def:qdf}
Let $f$ be a negative and-net and $S$ be a set of cycles of $\G(f)$. An \emph{$S$-quasi-delocalizing function} of $f$ is a function $\chi$ from $S$ to the set of pairs of edges of $\G(f)$ such that:
\begin{itemize}
\item $\chi_1(C)$ is a chord $(i,k)$ of $C$;
\item $\chi_2(C)$ is an edge $(i,j)$ of $C$;
\item $\im(\chi_1)\cap\im(\chi_2)=\varnothing$.
\end{itemize}
\end{definition}
Note that in this definition, $\chi_2(C)$ is determined by $\chi_1(C)$: it is the unique edge of $C$ starting from the same vertex as $\chi_1(C)$. But it is more convenient to define $\chi(C)$ as a pair of edges.
\begin{lemma}
\label{trick}
Let $f$ be a negative and-net and $S$ be a set of cycles of $\G(f)$. If $f$ has an $S$-quasi-delocalizing function, then $f$ can be expanded to an and-net $g$ such that every cycle of $\G(g)$ above a cycle of $S$ is delocalized.
\end{lemma}
\begin{figure}[!t]
\centering
\begin{pspicture}(12, 3)
\psset{nodesep=2pt,arrowsize=6pt,labelsep=3pt}
\put(0,1.5){\rnode{V1}{$j$}}
\nccircle[angle=-90,linestyle=dashed,linewidth=0.3pt,nodesep=4pt]{-}{V1}{1.5}\ncput*[npos=.40]{\rnode{V2}{$k$}}\ncput*[npos=.75]{\rnode{U}{$i$}}\ncput[npos=.25]{$\pmb{>}$}
\ncline{-|}{U}{V2}
\ncarc[arcangleA=-45,arcangleB=-41]{-|}{U}{V1}
\put(4.3,1.5){\rnode{VV1}{$j$}}
\nccircle[angle=-90,linestyle=dashed,linewidth=0.3pt,nodesep=4pt]{-}{VV1}{1.5}\ncput*[npos=.40]{\rnode{VV2}{$k$}}\ncput*[npos=.75]{\rnode{UU}{$i$}}\ncput[npos=.25]{$\pmb{>}$}\ncput*[npos=.875]{\rnode{VV1a}{$i'$}}
\ncline{-|}{UU}{VV2}
\ncarc[arcangleA=-17,arcangleB=-15]{->}{UU}{VV1a}
\ncarc[arcangleA=-15,arcangleB=-15]{-|}{VV1a}{VV1}
\put(8.8,1.5){\rnode{VVV1}{$j$}}
\nccircle[angle=-90,linestyle=dashed,linewidth=0.3pt,nodesep=4pt]{-}{VVV1}{1.5}\ncput*[npos=.40]{\rnode{VVV2}{$k$}}\ncput*[npos=.75]{\rnode{UUU}{$i$}}\ncput[npos=.25]{$\pmb{>}$}\ncput*[npos=.875]{\rnode{VVV1a}{$i'$}}
\ncline{-|}{UUU}{VVV2}
\ncarc[arcangleA=-17,arcangleB=-15]{->}{UUU}{VVV1a}
\ncarc[arcangleA=-15,arcangleB=-15]{-|}{VVV1a}{VVV1}
\put(10.1,1.5){\rnode{VVV1b}{$i''$}}
\ncline[nodesepB=2pt]{->}{UUU}{VVV1b}
\ncline[nodesepA=1pt,nodesepB=0pt]{->}{VVV1b}{VVV1a}
\ncarc[arcangleA=-30,arcangleB=-30,nodesepA=0pt,nodesepB=4pt]{-|}{VVV1b}{VVV2}
\end{pspicture}
\caption{The trick of Lemma \ref{trick}.}
\label{fig:trick}
\end{figure}
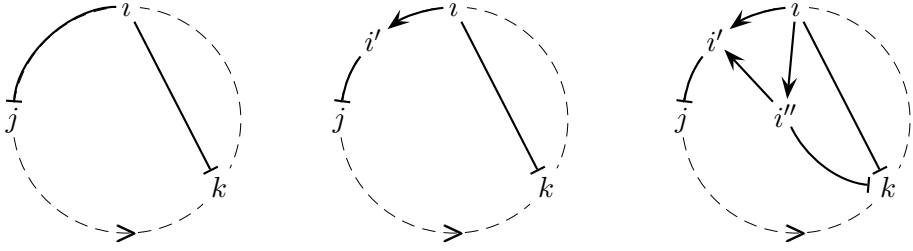
\begin{proof}
Let $\chi$ be an $S$-quasi-delocalizing function of $f$. We proceed in two steps, as illustrated in Figure \ref{fig:trick}. We first define an and-net $f'$ by replacing in $\G(f)$ each edge $(i,j)\in\im(\chi_2)$ by two edges $(i,i'), (i',j)$, where $i'$ is a new vertex, $(i,i')$ is positive and $(i',j)$ is negative. Since $\im(\chi_1)\cap\im(\chi_2)=\varnothing$, $\im(\chi_1)$ is a set of negative egdes of $f'$.

We then define $g$ by adding to $f'$, for each $(i,k)\in\im(\chi_1)$, three edges $(i,i''),(i'',i'),(i'',k)$, where $i''$ is a new vertex, $(i,i''),(i'',i')$ are positive and $(i'',k)$ is negative.

Now, $f'$ reduces to $f$ and $g$ reduces to $f'$, so these two steps are expansions, as required. Finally, let
$$
C=(i,j,P,k,Q,i)
$$
be a cycle of $S$, where $\chi(C)=((i,k),(i,j))$ and $P,Q$ are paths in $\G(f)$. A cycle of $\G(g)$ which is above $C$ (Section \ref{sec:red}) contains the vertices $i,j,k$ and is of the form
$$
\text{either } (i,i',j,P',k,Q',i) \text{ or } (i,i'',i',j,P',k,Q',i),
$$
for some paths $P',Q'$ in $\G(g)$. Note that $P',Q'$ may differ from $P,Q$ because other edges than $(i,k)$ and $(i,j)$ may belong to $\im(\chi_1)\cup\im(\chi_2)$.

In the first case, the cycle is delocalized by the triple $(i'',i',k)$. In the second case, it is delocalized by $(i,i',k)$.
\end{proof}
\begin{lemma}
\label{trickpol}
Let $f$ be a negative and-net and $S$ be the set of positive (\resp negative) cycles of $\G(f)$. If $f$ has an $S$-quasi-delocalizing function, then $f$ can be expanded to an and-net without local positive (\resp negative) cycle.
\end{lemma}
\begin{proof}
Let $g$ be the and-net given by expansion of $f$ in Lemma \ref{trick}. Each positive (\resp negative) cycle of $\G(g)$ is above some cycle of $\G(f)$ with the same sign, thus above a cycle of $S$: it is therefore delocalized.
\end{proof}

\subsection{Definition of a counter-example}
\label{sec:dfcount}
Let us now prove Theorem A.

\begin{figure}[!t]
\centering
\psset{unit=10mm}
\begin{pspicture}(-9,-3)(3,3)
\psset{nodesep=2pt,arrows=->,arrowsize=6pt,labelsep=1pt}
\SpecialCoor
\rput(3;90){\rnode{0}{\small0}}
\rput(3;135){\rnode{1}{\small5}}
\rput(3;180){\rnode{2}{\small1}}
\rput(3;225){\rnode{3}{\small7}}
\rput(3;270){\rnode{4}{\small2}}
\rput(3;315){\rnode{5}{\small9}}
\rput(3;360){\rnode{6}{\small3}}
\rput(3;405){\rnode{7}{\small11}}
\rput(2.1;125){\rnode{8}{\small4}}
\rput(2.1;215){\rnode{9}{\small6}}
\rput(2.1;305){\rnode{10}{\small8}}
\rput(2.1;395){\rnode{11}{\small10}}
\ncline{0}{1}\ncline{0}{8}\ncline{8}{1}\ncline[arrows=-|]{1}{2}
\ncline{2}{3}\ncline{2}{9}\ncline{9}{3}\ncline[arrows=-|]{3}{4}
\ncline{4}{5}\ncline{4}{10}\ncline{10}{5}\ncline[arrows=-|]{5}{6}
\ncline{6}{7}\ncline{6}{11}\ncline{11}{7}\ncline[arrows=-|]{7}{0}
\psset{arrows=-|,nodesepB=5pt,arcangleA=-20,arcangleB=-30}
\ncarc{8}{4}\ncarc{9}{6}\ncarc{10}{0}\ncarc{11}{2}
\psset{arcangle=-15,nodesepA=5pt}
\ncarc{0}{4}\ncarc{2}{6}\ncarc{4}{0}\ncarc{6}{2}
\rput(-6.5,0){
\rput(2;90){\rnode{0a}{\small0}}
\rput(2;180){\rnode{1a}{\small1}}
\rput(2;270){\rnode{2a}{\small2}}
\rput(2;360){\rnode{3a}{\small3}}}
\psset{arrows=-|,nodesep=3pt,arcangle=-45}
\ncarc{0a}{1a}\ncarc{1a}{2a}\ncarc{2a}{3a}\ncarc{3a}{0a}
\psset{arcangle=-18,nodesepA=5pt}
\ncarc{0a}{2a}\ncarc{2a}{0a}\ncarc{1a}{3a}\ncarc{3a}{1a}
\end{pspicture}
\caption{The and-nets $f,g$ of Theorem A. On the right, $g$ is a fixed-point-free and-net without local negative cycle.}
\label{fig:ce}
\end{figure}
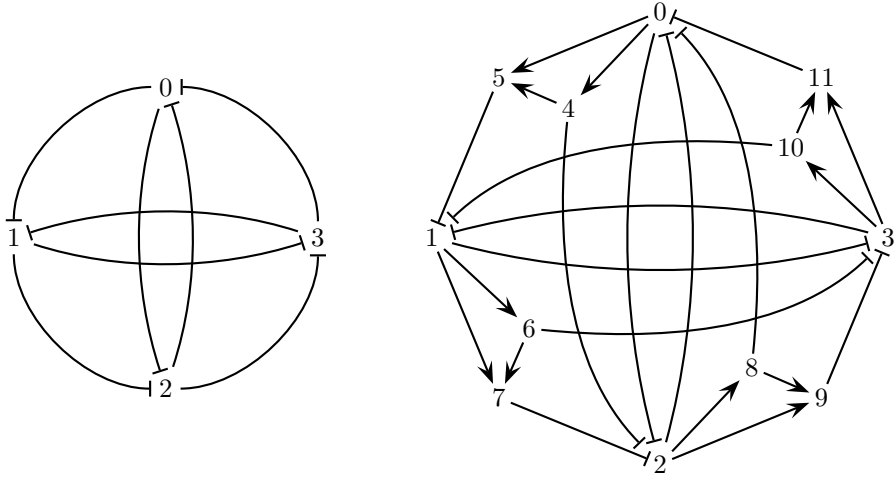
Let $f$ be the negative and-net defined on the left side of Figure \ref{fig:ce}. The transpose of the underlying directed graph is an example of graph without kernel, whose odd cycles all have a chord. The set $S$ of negative cycles of $\G(f)$ consists of the $4$ cycles $C_i=(i,i+1,i+2,i)$ for $i\in\{0,1,2,3\}$ (where numbers are taken modulo $4$). The unique $S$-quasi-delocalizing function $\chi$ is given by
\begin{align*}
\chi_1(C_i) &= (i,i+2) \\
\chi_2(C_i) &= (i,i+1).
\end{align*}
By Lemma \ref{trickpol}, $f$ can then be expanded to an and-net $g$ without local negative cycle, which is pictured on the right side of Figure \ref{fig:ce}.

On the other hand, $f$ has no fixed point: indeed, $(0,0,0,0)$ is clearly not a fixed point of $f$, and if $x\in\B^4$ is fixed point such that $x_i=1$, then $x_{i+1}=x_{i+2}=0$, whence $x_{i+3}=1$ and $f_i(x)=0\neq x_i$, contradiction. Since $g$ is an expansion of $f$, by Proposition \ref{prop:red}, it has no fixed point either. This concludes the proof of Theorem A.

Notice that some negative cycles of $\G(g)$ have only external delocalizing triples, for instance $(0,5,1,7,2,0)$ has an external delocalizing triple $(4,5,2)$, and no internal one. The network $g$ is therefore not in contradiction with point 4 of Theorem \ref{th:knownneg}.

On the other hand, the $4$-dimensional negative and-net $f$ of Figure \ref{fig:ce} has an attractive cycle
$$
\begin{array}{@{\extracolsep{6mm}}ccccccccc}
&&&&&&&& \\
\rnode{1}{e^{3}} &
\rnode{2}{e^{2,3}} &
\rnode{3}{e^{2}} &
\rnode{4}{e^{1,2}} &
\rnode{5}{e^{1}} &
\rnode{6}{e^{0,1}} &
\rnode{7}{e^{0}} &
\rnode{8}{e^{3,0}} &
\rnode{9}{e^{3}}
\psset{nodesep=3pt,arrowsize=5pt,linewidth=0.5pt,labelsep=3pt}
\ncline[arrows=->]{1}{2}\naput{\scriptstyle 2}
\ncline[arrows=->]{2}{3}\naput{\scriptstyle 3}
\ncline[arrows=->]{3}{4}\naput{\scriptstyle 1}
\ncline[arrows=->]{4}{5}\naput{\scriptstyle 2}
\ncline[arrows=->]{5}{6}\naput{\scriptstyle 0}
\ncline[arrows=->]{6}{7}\naput{\scriptstyle 1}
\ncline[arrows=->]{7}{8}\naput{\scriptstyle 3}
\ncline[arrows=->]{8}{9}\naput{\scriptstyle 0}
\end{array}
$$
where the number above an arrow indicates the updated coordinate. This attractive cycle $\theta$ gives rise to a cyclic attractor in the expanded $12$-dimensional and-net $g$ of Figure \ref{fig:ce}, which is obtained by replacing each trajectory of $\theta$ of the form $(e^{i},e^{i-1,i},e^{i-1})$ by two trajectories from $e^{i,4+2i,5+2i}$ to $e^{i-1,2+2i,3+2i}$. For instance, $(e^3,e^{2,3},e^2)$ is replaced by the two trajectories
$$
\begin{array}{@{\extracolsep{5mm}}ccccc}
&&&& \\
\rnode{2}{e^{2,3,10,11}} &
\rnode{1}{e^{3,10,11}} &&& \\
&&&& \\
&&& \rnode{6}{e^{2,8,9,11}} & \\
\rnode{3}{e^{2,8,3,10,11}} &
\rnode{4}{e^{2,8,9,3,10,11}} &
\rnode{5}{e^{2,8,9,10,11}} &&
\rnode{8}{e^{2,8,9}} \\
&&& \rnode{7}{e^{2,8,9,10}} & \\
\psset{nodesep=2pt,arrowsize=5pt,linewidth=0.5pt,labelsep=3pt}
\ncline[arrows=->]{1}{2}\nbput{\scriptstyle 2}
\ncline[arrows=->]{2}{3}\naput{\scriptstyle 8}
\ncline[arrows=->]{3}{4}\naput{\scriptstyle 9}
\ncline[arrows=->]{4}{5}\naput{\scriptstyle 3}
\ncline[arrows=->]{5}{6}\naput{\scriptstyle 10}
\ncline[arrows=->]{5}{7}\nbput{\scriptstyle 11}
\ncline[arrows=->]{6}{8}\naput{\scriptstyle 11}
\ncline[arrows=->]{7}{8}\nbput{\scriptstyle 10}
\end{array}
$$
and similarly for the trajectories $(e^{2},e^{1,2},e^{1}),(e^{1},e^{0,1},e^{0})$ and $(e^{0},e^{3,0},e^{3})$ of $\theta$. This cyclic attractor is therefore not an attractive cycle, but almost in a certain sense: adding new variables $4+2i$ and $5+2i$ has delocalized all negative cycles, but decreasing $i$ now forces the two updates of $4+2i$ and $5+2i$ at the same time, whence a critical pair which is immediately convergent.

We do not know whether and-nets with an attractive cycle must have a local negative cycle, but we shall see in Section \ref{sec:attcyc} that, in general, arbitrary Boolean networks may have an attractive cycle and no local negative cycle.

\subsection{Reformulation in terms of kernels}
\label{sec:ker}
Let us first insert a consequence of Theorem A in graph theory.

Let $G$ be a directed graph. Given vertices $u,v$ of $G$ (not necessarily distinct), recall from \cite{RR12} that a vertex $w\neq u,v$ is said to be a \emph{subdivision of $(u,v)$} when $(u,w)$ and $(w,v)$ are arcs of $G$, $(u,v)$ is not an arc of $G$, and the in-degree and out-degree of $w$ both equal $1$. A vertex is called a {\emph{subdivision}} when it is a subdivision of some pair of vertices. Given a cycle $C$ of $G$ and vertices $u,v_1,v_2$ of $G$, $(u,v_1,v_2)$ is called a \emph{killing triple} of $C$ when:
\begin{itemize}
\item $v_1$ and $v_2$ are distinct vertices of $C$,
\item $(v_1,u)$ has a subdivision in $G$, but no subdivision of $(v_1,u)$ belongs to $C$.
\item $(v_2,u)$ is an arc of $G$ that is not in $C$,
\end{itemize}
A killing triple $(u,v_1,v_2)$ of $C$ is \emph{internal} when $u$ is a vertex of $C$.

Killing triples mimic delocalizing triples by replacing the positive edge by a pair of consecutive edges through a new point, the subdivision.

In \cite{RR12}, we proved that directed graphs in which every odd cycle has an internal killing triple must have a kernel. This may be contrasted with the following result, which is as an immediate consequence of Theorem A.
\begin{thmaa}
There exist kernel-free directed graphs in which every odd cycle has a killing triple.
\end{thmaa}

\section{Second construction: Boolean networks with attractive cycles}
\label{sec:attcyc}

\begin{thmb}
There exist Boolean networks with no local negative cycle and an attractive cycle.
\end{thmb}
To prove Theorem B, we shall start with a Boolean network with an antipodal attractive cycle, and then modify the neighborhood of this attractive cycle so as to delocalize all negative cycles. We explain this delocalizing process in Section \ref{sec:pad} and construct the actual counter-example in Section \ref{sec:def}.

\subsection{Padding around an attractive cycle}
\label{sec:pad}
We begin with a remark in \cite{Rue14}: the fact that a Boolean network $f:\B^n\rightarrow\B^n$ has an attractive cycle $\theta$, even an antipodal one, does not imply that for some $x$ on the cycle $\theta$, $\G(f)(x)$ has a negative cycle. A counterexample $f$ to this statement is defined in \cite{Rue14} by starting with an antipodal attractive cycle
$$
\theta=(0,\ldots,e^{1,\ldots,n-1},\overline{0},\ldots,\overline{e^{1,\ldots,n-1}},0)
$$
and adding \emph{moves} (directed edges of $\Gamma(f)$) to delocalize negative cycles at points of $\theta$. The network solely consisting of $\theta$ and no other moves has indeed many small negative cycles: a negative cycle $(i,i+1,i)$ in $\G(f)(e^{1,\ldots,i-1})$ and $\G(f)(\overline{e^{1,\ldots,i-1}})$ for each $i$. The way they are delocalized in \cite{Rue14} creates new local negative cycles (outside of $\theta$), so we may wonder if there exist alternatives. 

The following lemma shows that the first steps of this delocalization process amount to essentially two choices.
\begin{lemma}
\label{lem:negcyc}
Let $k\geqslant 2$, $n\geqslant k+2$ and $f:\B^n\rightarrow\B^n$ be a Boolean network such that for each $i\in\{1,\ldots,k\}$, $f(e^{1,\ldots,i-1})=e^{1,\ldots,i}$, where $e^{1,\ldots,i-1}=0$ when $i=1$. If $\G(f)$ has no local negative cycle, then for each $i\in\{1,\ldots,k\}$, $\Gamma(f)$ contains as a subgraph one of the two graphs $H_i,K_i$ of Figure \ref{fig:sol}.
\end{lemma}
\begin{figure}[!t]
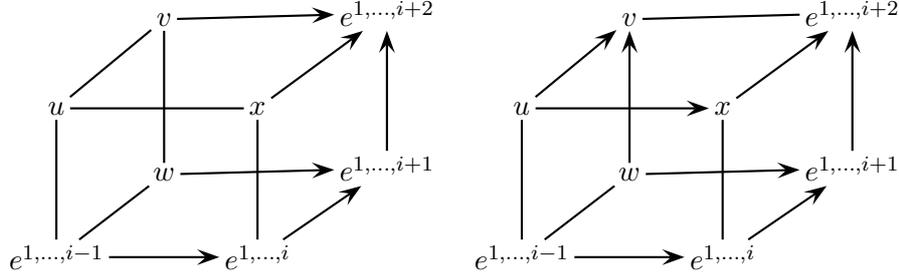

$$
\begin{array}{cc}
\begin{array}{@{\extracolsep{3mm}}cccc}
& \rnode{011}{v} && \rnode{111}{e^{1,\ldots,i+2}} \\[7mm]
\rnode{001}{u} && \rnode{101}{x} & \\[4mm]
& \rnode{010}{w} && \rnode{110}{e^{1,\ldots,i+1}} \\[7mm]
\rnode{000}{e^{1,\ldots,i-1}} && \rnode{100}{e^{1,\ldots,i}} &
\psset{nodesep=2pt,arrowsize=6pt}
\ncline[arrows=->]{000}{100}
\ncline[arrows=->]{100}{110}
\ncline[arrows=->]{110}{111}
\ncline[arrows=->]{010}{110}
\ncline[arrows=->]{101}{111}
\ncline[arrows=-]{000}{010}
\ncline[arrows=-]{000}{001}
\ncline[arrows=<-]{111}{011}
\ncline[arrows=-]{010}{011}
\ncline[arrows=-]{011}{001}
\ncline[arrows=-]{001}{101}
\ncline[arrows=-]{101}{100}
\end{array}
&
\begin{array}{@{\extracolsep{3mm}}cccc}
& \rnode{011}{v} && \rnode{111}{e^{1,\ldots,i+2}} \\[7mm]
\rnode{001}{u} && \rnode{101}{x} & \\[4mm]
& \rnode{010}{w} && \rnode{110}{e^{1,\ldots,i+1}} \\[7mm]
\rnode{000}{e^{1,\ldots,i-1}} && \rnode{100}{e^{1,\ldots,i}} &
\psset{nodesep=2pt,arrowsize=6pt}
\ncline[arrows=->]{000}{100}
\ncline[arrows=->]{100}{110}
\ncline[arrows=->]{110}{111}
\ncline[arrows=->]{010}{110}
\ncline[arrows=->]{101}{111}
\ncline[arrows=-]{000}{010}
\ncline[arrows=-]{000}{001}
\ncline[arrows=-]{111}{011}
\ncline[arrows=->]{010}{011}
\ncline[arrows=<-]{011}{001}
\ncline[arrows=->]{001}{101}
\ncline[arrows=-]{101}{100}
\end{array}
\end{array}
$$
\caption{The directed edges are the edges of the two subgraphs $H_i$ (left) and $K_i$ (right) of Lemma \ref{lem:negcyc}. In these pictures, undirected edges keep track of the hypercube structure of $\B^n$ but do not belong to $H_i$ or $K_i$.}
\label{fig:sol}
\end{figure}
This only means that the directed edges of $H_i$ or $K_i$ should be edges of $\Gamma(f)$ (and not that $H_i$ or $K_i$ should be an induced subgraph of $\Gamma(f)$).
\begin{proof}
$\Gamma(f)$ must contain the two edges
$$
(w, e^{1,\ldots,i+1}) \text{ and } (x, e^{1,\ldots,i+2})
$$
(see Figure \ref{fig:sol}) because otherwise $(i,i+1,i)$ is a negative cycle of $\G(f)(e^{1,\ldots,i-1})$ or $(i+1,i+2,i+1)$ is a negative cycle of $\G(f)(e^{1,\ldots,i})$.

Furthermore, $\Gamma(f)$ contains
$$
\text{either } (v, e^{1,\ldots,i+2}) \text{ or } (w,v),
$$
since otherwise, $(i,i+2,i)$ is a negative cycle of $\G(f)(w)$. In the first case, $\Gamma(f)$ contains $H_i$ as a subgraph. In the second case, $\Gamma(f)$ must contain $(u,x)$, because otherwise, $(i,i+1,i+2,i)$ is a negative cycle of $\G(f)(e^{1,\ldots,i-1})$. And finally, $\Gamma(f)$ must contain
$$
\text{either } (v, e^{1,\ldots,i+2}) \text{ or } (u,v),
$$
because otherwise, $(i,i+1,i)$ is a negative cycle of $\G(f)(u)$. Thus $\Gamma(f)$ contains a supergraph of $H_i$ again in the first case, and $K_i$ in the second case.
\end{proof}
Containing $H_i$ or $K_i$ as a subgraph for all $i$ is certainly not a sufficient condition for $f$ to have no local negative cycle. For instance, the network with an antipodal attractive cycle which is defined in \cite{Rue14} is obtained by adding, for all $x$ on the cycle such that $f(x)=x+e^i$ and all $j\neq i$, the edges $(x+e^j,x)$. In particular, $H_i$ is the systematic choice, and the resulting network still has local negative cycles.

In general, starting with an antipodal attractive cycle $\theta$, the consecutive choices of $H_i$ or $K_i$ delocalize some negative cycles by creating new ones. The question is therefore whether this non-deterministic process of padding the asynchronous dynamics with new moves can terminate by delocalizing all negative cycles.

Actually, it suffices to know how to delocalize negative cycles starting from a linear deterministic trajectory: the fact that $\theta$ above is a cycle is irrelevant. This is simply because interaction graphs are defined locally, and for $n$ large enough, an antipodal attractive cycle looks locally like a linear trajectory. Indeed, let $N(X,r)$ denote, for any $X\subseteq\B^n$, the \emph{$r$-neighborhood} of $X$ (the set of points $x$ such that $d(x,X)\leqslant r$), and assume that, for some constant $r$, we have a procedure for delocalizing all negative cycles of an arbitrarily long linear trajectory $L=(0,\ldots,e^1,e^{1,2},e^{1,\ldots,n-1})$ by adding moves starting only at points of $N(L,r)$, so that points outside $N(L,r)$ are fixed. Assume moreover that this procedure is homogeneous enough: the choices of $H_i$ or $K_i$ are periodic (of period independent of $n$), so that the procedure does not depend too much on the position in $L$. By using this procedure, we may now delocalize all negative cycles of a Boolean network $f$ with an antipodal attractive cycle $\theta$, for $n$ large enough. First notice that, by the following lemma \cite{Rue14}, no negative cycle may be localized at points outside $N(\theta,r)$.
\begin{lemma}
\label{lem:sgnpar}
If $C$ is a cycle of $\G(f)(x)$ with vertex set $I$, then $C$ is positive (resp. negative) when $x$ has an even (resp. odd) out-degree in $\Gamma(f)$. In particular, if $x$ is a fixed point, $\G(f)(x)$ has no negative cycle.
\end{lemma}
Besides, the interaction graph $\G(f)(y)$ depends on moves starting only from points in $N(\{y\},1)$, and if $y\in N(\theta,r)$, we have for $n$ sufficiently larger than $r$:
$$
N(\{y\},1)\cap N(\theta,r)=N(\{y\},1)\cap N(L,r)
$$
for some linear portion $L$ of $\theta$ which contains $y$. Intuitively, the context $N(\{y\},1)$ useful for $\G(f)(y)$ only sees a linear part of the cycle. See Figure \ref{fig:padding}.

We shall now make this intuition precise by showing that systematically choosing $K_i$-type graphs and padding $N(\theta,r)$ up to $r=2$ delocalizes all negative cycles if $n\geqslant 7$.
\begin{figure}[!t]
\SpecialCoor
\centering
\psset{unit=10mm}
\begin{pspicture}(-3.3,-3.3)(4.3,3.3)
\psset{nodesep=2pt,arrowsize=7pt}
\put(0,0){\rnode{C}{}}
\pscircle[linewidth=1pt](C){2.5}
\pscircle[linestyle=dashed,linewidth=0.5pt](C){1.7}
\pscircle[linestyle=dashed,linewidth=0.5pt](C){3.3}
\rput(0,2.5){$\pmb{<}$}
\rput*(2.5;180){$\theta$}
\rput*(2.5;-45){\small$a^i$}\pscircle[linewidth=0.5pt](2.5;-45){0.4}\pscircle[linewidth=0.5pt](2.5;-45){0.8}
\rput*(2.1;45){\small$b^j$}\pscircle[linewidth=0.5pt](2.1;45){0.4}
\rput*(2.9;135){\small$c^k$}\pscircle[linewidth=0.5pt](2.9;135){0.4}
\rput*(1.7;-135){\small$d^\ell$}\pscircle[linewidth=0.5pt](1.7;-135){0.4}
\pscurve{->}(4.5;0)(4;-10)(2.7;-28)
\rput(5;4){\small$A^i\subseteq N(\{a^i\},2)$}
\end{pspicture}
\caption{Padding around an attractive cycle $\theta$ in Sections \ref{sec:pad} and \ref{sec:noneg}. The $2$ large dashed circles enclose $N(\theta,2)$, and the $4$ small circles represent $N(\{a^i\},1),N(\{b^j\},1),N(\{c^k\},1)$ and $N(\{d^\ell\},1)$.}
\label{fig:padding}
\end{figure}
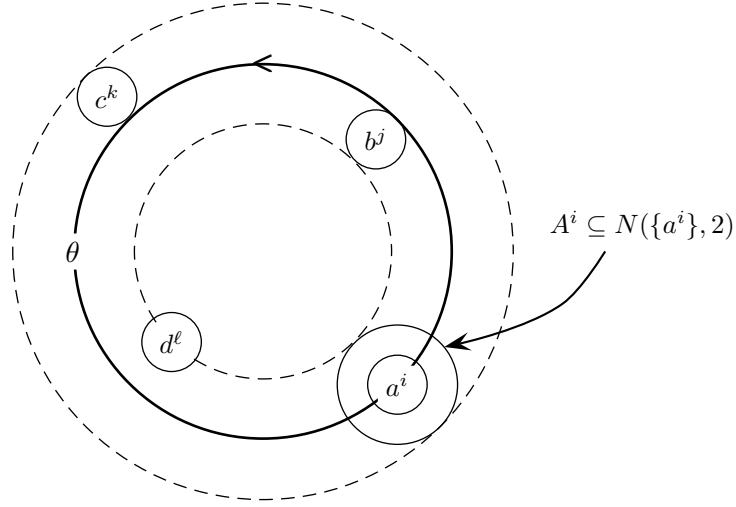

\subsection{Definition of a counter-example}
\label{sec:def}
We now assume $n\geqslant 7$.

Consider the Boolean network $f:\B^n\rightarrow\B^n$ defined as follows. For $1\leqslant i\leqslant n$, let
$$
a^i=e^{1,\ldots,i-1} \text{ and } a^{n+i}=\overline{a^i}.
$$
Hence $a^1=0$, and the antipodal attractive cycle $\theta$ is $(a^1,\ldots,a^{2n},a^1)$. Let $e^{i+kn}=e^i$ for $i\geqslant 1$ and any $k\in\Z$, and for $1\leqslant i\leqslant 2n$, let
\begin{align*}
b^i &= a^i+e^{i+1} \\
c^i &= a^i+e^{i+2} \\
d^i &= a^i+e^{i+2,i+3},
\end{align*}
so that $b^{n+i}=\overline{b^i},c^{n+i}=\overline{c^i}$ and $d^{n+i}=\overline{d^i}$. To simplify later notations, these four sequences of points of length $2n$ are extended to $\Z$-indexed sequences by letting $a^{i+2kn}=a^i$ for $1\leqslant i\leqslant 2n$ and any $k$, and similarly for $(b^i)_{i\in\Z},(c^i)_{i\in\Z},(d^i)_{i\in\Z}$. For any $i\in\Z$, let also:
$$
A^i = \{a^i,b^i,c^i,d^i\}, \quad A=\bigcup_1^{2n} A^i.
$$
See Figure \ref{fig:padding}.
\begin{itemize}
\item For $1\leqslant i\leqslant 2n$, define $f(a^i)=a^{i+1}$: in particular, $f(a^{2n})=a^1$ and $\theta$ is an antipodal attractive cycle of $f$.
\item For $1\leqslant i\leqslant 2n$, define
\begin{align*}
f(b^i) &= a^{i+3} \\
f(c^i) &= a^{i+3} \\
f(d^i) &= a^{i+4}+e^{i+1}.
\end{align*}
In other terms:
\begin{align*}
f(b^i)+b^i &= e^{i,i+2} \\
f(c^i)+c^i &= e^{i,i+1} \\
f(d^i)+d^i &= e^i,
\end{align*}
therefore, for $1\leqslant i\leqslant n$, $b^i$ and its antipode $b^{n+i}$ have $2$ degrees of freedom in $\Gamma(f)$, $i$ and $i+2$, $c^i$ and its antipode have $2$ degrees of freedom, $i$ and $i+1$, and $d^i$ and its antipode have $1$ degree of freedom, $i$.
\item Any other point $x\in\B^n$ is fixed: $f(x)=x$.
\end{itemize}
\begin{figure}[!t]
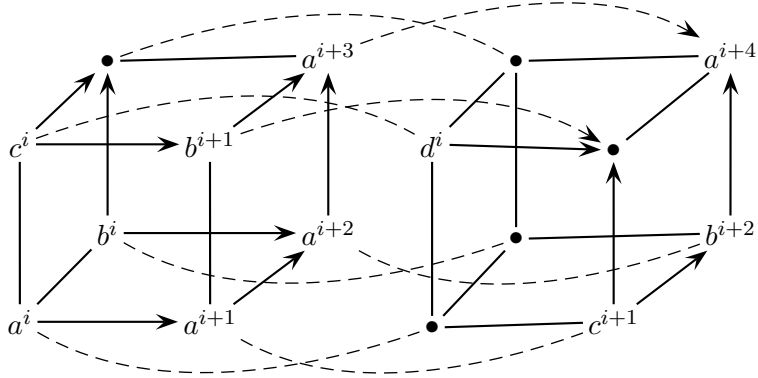

$$
\hspace{8mm}
\begin{array}{@{\extracolsep{5mm}}ccccccccc}
& \rnode{011}{\bullet} && \rnode{111}{a^{i+3}} &
& \rnode{q011}{\bullet} && \rnode{q111}{a^{i+4}} & \\[7mm]
\rnode{001}{c^i} && \rnode{101}{b^{i+1}} &&
\rnode{q001}{d^i} && \rnode{q101}{\bullet} && \\[7mm]
& \rnode{010}{b^i} && \rnode{110}{a^{i+2}} &
& \rnode{q010}{\bullet} && \rnode{q110}{b^{i+2}} & \\[7mm]
\rnode{000}{a^i} && \rnode{100}{a^{i+1}} &&
\rnode{q000}{\bullet} && \rnode{q100}{c^{i+1}} && \\
\psset{nodesep=2pt,arrowsize=6pt}
\ncline[arrows=->]{000}{100}
\ncline[arrows=->]{100}{110}
\ncline[arrows=->]{110}{111}
\ncline[arrows=->]{010}{110}
\ncline[arrows=->]{101}{111}
\ncline[arrows=-]{000}{010}
\ncline[arrows=-]{000}{001}
\ncline[arrows=-]{111}{011}
\ncline[arrows=->]{010}{011}
\ncline[arrows=<-]{011}{001}
\ncline[arrows=->]{001}{101}
\ncline[arrows=-]{101}{100}
\ncline[arrows=-]{q000}{q100}
\ncline[arrows=->]{q100}{q110}
\ncline[arrows=->]{q110}{q111}
\ncline[arrows=-]{q010}{q110}
\ncline[arrows=-]{q101}{q111}
\ncline[arrows=-]{q000}{q010}
\ncline[arrows=-]{q000}{q001}
\ncline[arrows=-]{q111}{q011}
\ncline[arrows=-]{q010}{q011}
\ncline[arrows=-]{q011}{q001}
\ncline[arrows=->]{q001}{q101} 
\ncline[arrows=<-]{q101}{q100}
\psset{arcangleA=-30,arcangleB=-20,linestyle=dashed,linewidth=0.5pt}
\ncarc[arrows=-]{000}{q000}
\ncarc[arrows=-]{100}{q100}
\ncarc[arrows=-]{010}{q010}
\ncarc[arrows=-]{110}{q110}
\psset{arcangleA=20,arcangleB=30}
\ncarc[arrows=-]{001}{q001}
\ncarc[arrows=->]{101}{q101}
\ncarc[arrows=-]{011}{q011}
\ncarc[arrows=->]{111}{q111}
\end{array}
$$
\caption{Partial illustration of the definition of $f$ in Section \ref{sec:def}. Dashed edges in dimension $i+3$ are simply meant to facilitate visualization.}
\label{fig:def}
\end{figure}
This definition is illustrated in Figure \ref{fig:def}. Let us make a few comments. First, the points $b^i,c^i$ with $2$ degrees of freedom correspond to the choice of subgraphs $K_i$ (Figure \ref{fig:sol}). Then the overlapping of these successive subgraphs $K_i$ creates, for each $i$, $1\leqslant i\leqslant n$, a negative cycle $(i,i+3,i)$ in $\G(f)(c^i)$ and $\G(f)(c^{n+i})$: the purpose of the degree of freedom $i$ of the $2n$ points $d^i$ is to delocalize these negative cycles. Also, as it is announced, only points at distance at most $2$ from $\theta$ are not fixed points.

To prove that the above Boolean network $f$ is indeed well-defined (Proposition \ref{prop:wd}), we shall need a few remarks on the isometries of $\B^n$ to reduce the number of necessary verifications. It is clear that these isometries are exactly the graph automorphisms of the hypercube $\B^n$, and well-studied as such. See, e.g., \cite{Har00} and references therein. It is not difficult to figure out either that they are precisely the functions $U$ from $\B^n$ to itself such that, for any $I\subseteq\{1,\ldots,n\}$:
$$
U(e^I)=U_0+e^{\sigma(I)}
$$
for some permutation $\sigma\in\Sig_n$ and some $U_0\in\B^n$. Here, $\sigma(I)=\{\sigma(i),i\in I\}$. We check this in \ref{sec:isomproof} to make the paper self-contained. So, for any $x\in\B^n$, $U(x+e^I)=U(x)+e^{\sigma(I)}$.

Let $S:\B^n\rightarrow\B^n$ be the cyclic permutation of coordinates defined by $S(x)=(x_n,x_1,\ldots,x_{n-1})$ and $T$ be the fixed-point-free isometry of $\B^n$ defined by $T(x)=S(x)+e^1$.

Since $T(x+e^i)=T(x)+e^{i+1}$ for any $x$ and $i$, we have the following:
\begin{lemma}
\label{lem:abcdequiv}
For any $i$, $T(a^i)=a^{i+1}, T(b^i)=b^{i+1}, T(c^i)=c^{i+1}, T(d^i)=d^{i+1}$.
\end{lemma}
\begin{proposition}
\label{prop:wd}
If $n\geqslant 7$ and $1\leqslant i,j,k,\ell\leqslant 2n$ are all different, then so are $a^i,b^j,c^k,d^\ell$. Therefore, the Boolean network $f$ is well-defined.
\end{proposition}
To prove Proposition \ref{prop:wd}, it suffices, by Lemma \ref{lem:abcdequiv}, to compute the distances between points of $A^0$ and points of $A^i$ for any $i\neq 0$. This is postponed to \ref{sec:appen}.

Now, to prove Theorem B, it remains to show that $\G(f)$ has no local negative cycle. Since all points outside $A=\bigcup_1^{2n} A^i$ are fixed points, Lemma \ref{lem:sgnpar} shows that no negative cycle may be localized at these points. Therefore it suffices to check that $\G(f)(x)$ has no negative cycle for $x\in A$. By Lemmas \ref{lem:equiv} and \ref{lem:fequiv} below, it suffices to check this for $x\in A^0$.

\subsection{Equivariance}
\label{sec:equivariance}
\begin{lemma}
\label{lem:equiv}
If $U$ is an isometry of $\B^n$ and $f:\B^n\rightarrow\B^n$ is $U$-equivariant ($f\circ U=U\circ f$), then for any $x$, $\G(f)(U(x))$ and $\G(f)(x)$ have isomorphic underlying directed graphs. Moreover, corresponding cycles have the same sign.
\end{lemma}
\begin{proof}
Let $\sigma\in\Sig_n$ be the permutation associated with $U$. Then:
\begin{align*}
\partial_if(U(x)) &= f(U(x))+f(U(x)+e^i) \\
&= f(U(x))+f(U(x+e^k)) \quad \text{for $k=\sigma^{-1}(i)$} \\
&= U(f(x))+U(f(x+e^k)) \quad \text{by equivariance} \\
&= U(f(x))+U(f(x)+e^J) \quad \text{where $e^J=\partial_kf(x)$} \\
&= e^{\sigma(J)}.
\end{align*}
Hence $(k,j)$ is an edge of $\G(f)(x)$ is and only if $(\sigma(k),\sigma(j))$ is an edge of $\G(f)(U(x))$: $\sigma$ is an isomorphism of directed graphs.

To show that it preserves the signs of cycles, let us start by observing that it preserves the degrees of freedom: if $f(x)+x=e^I$, then $U(f(x))+U(x)=e^{\sigma(I)}$. A cycle $C$ of $\G(f)(x)$ with vertex set $J$ corresponds to the cycle $\sigma(C)$ of $\G(f)(U(x))$ with vertex set $\sigma(J)$. By Lemma \ref{lem:sgnpar}, $C$ is positive if and only if $I\cap J$ has even cardinality, if and only if $\sigma(I)\cap\sigma(J)$ has even cardinality, \ie when $\sigma(C)$ is positive.
\end{proof}
\begin{lemma}
\label{lem:fequiv}
The Boolean network $f$ defined in Section \ref{sec:def} is $T$-equivariant.
\end{lemma}
\begin{proof}
We have:
\begin{align*}
(f\circ T)(a^i) &= f(a^{i+1})=a^{i+2}=T(a^{i+1})=(T\circ f)(a^i) \\
(f\circ T)(b^i) &= f(b^{i+1}) = a^{i+4} = T(a^{i+3}) = (T\circ f)(b^i),
\end{align*}
and similarly, $(f\circ T)(c^i)=(T\circ f)(c^i)$ and $(f\circ T)(d^i)=(T\circ f)(d^i)$. Finally, we have noticed in particular that
$$
T(A^i)=A^{i+1}
$$
for any $i$, whence $T(A)=A$. Therefore, if $x$ is a fixed point, so is $T(x)$, and then $(f\circ T)(x)=f(x)=(T\circ f)(x)$.
\end{proof}

\subsection{No local negative cycle}
\label{sec:noneg}
As a consequence of Lemmas \ref{lem:equiv} and \ref{lem:fequiv}, to prove Theorem B, it suffices to check that, when $n$ is large enough, for any $x\in A^0=\{a^0,b^0,c^0,d^0\}$, $\G(f)(x)$ has no negative cycle. Since the interaction graphs $\G(f)(x)$ depend on $N(x,1)$, we shall need the following list of all neighbors of points in $A^0$ which belong to $A$.
\begin{lemma}
\label{lem:neighbors}
If $n\geqslant 7$, then:
\begin{align*}
A\cap N(\{a^0\},1) &= \{a^{-1},a^0,a^1,b^{-2},b^0,c^0\} \\
A\cap N(\{b^0\},1) &= \{a^0,a^2,b^0,c^{-1}\} \\
A\cap N(\{c^0\},1) &= \{a^0,b^1,c^0,d^0\} \\
A\cap N(\{d^0\},1) &=
\begin{cases}
\{c^0,d^0\} & \text{ if $n\geqslant 8$} \\
\{c^0,d^0,d^{-5},d^{5}\} & \text{ if $n=7$.}
\end{cases}
\end{align*}
In other terms:
\begin{align*}
A\cap N(\{a^0\},1) &= a^0+\{0,e^{-2},e^{-1},e^0,e^1,e^2\} \\
A\cap N(\{b^0\},1) &= b^0+\{0,e^{-1},e^0,e^1\} = a^0+\{0,e^{-1,1},e^{0,1},e^1\} \\
A\cap N(\{c^0\},1) &= c^0+\{0,e^0,e^2,e^3\} = a^0+\{0,e^{0,2},e^2,e^{2,3}\} \\
A\cap N(\{d^0\},1) &=
\begin{cases}
d^0+\{0,e^3\} = a^0+\{e^2,e^{2,3}\} & \text{ if $n\geqslant 8$} \\
d^0+\{0,e^3,e^4,e^6\} = a^0+\{e^2,e^{2,3},e^{2,3,4},e^{2,3,6}\} & \text{ if $n=7$.}
\end{cases}
\end{align*}
\end{lemma}
Like the proof of Proposition \ref{prop:wd}, the proof of Lemma \ref{lem:neighbors} amounts to solve equations between sets of integers modulo $n$. Their common proof is postponed to \ref{sec:appen}.

By Lemma \ref{lem:neighbors}, we know $\G(f)(x)$ for any $x\in A^0$ and $n\geqslant 7$. We actually do not need a full computation of $\G(f)(x)$, but simply of its possible negative cycles: by Lemma \ref{lem:sgnpar}, these are cycles of $\G(f)(x)$ through an odd number of degrees of freedom of $x$.
\begin{description}
\item[($a^0$)] 
Since $f(a^0)=a^1=a^0+e^0$, a negative cycle of $\G(f)(a^0)$, if any, must pass through $0$. But $\partial_0f(a^0)=f(a^0)+f(a^1)=e^1$, hence in $\G(f)(a^0)$, $0$ has only one outgoing edge $(0,1)$. Then:
$$
\partial_1f(a^0)=f(a^0)+f(a^0+e^1)=f(a^0)+f(b^0)=a^1+a^3=e^{1,2},
$$
hence $1$ has only two outgoing edges $(1,1)$ and $(1,2)$. Finally, 
$$
\partial_2f(a^0)=f(a^0)+f(a^0+e^2)=f(a^0)+f(c^0)=a^1+a^3=e^{1,2},
$$
and $2$ has only two outgoing edges $(2,1)$ and $(2,2)$. Therefore no path in $\G(f)(a^0)$ starting from $0$ may loop, and $\G(f)(a^0)$ has no negative cycle.
\item[($b^0$)] 
Since $b^0$ has two degrees of freedom, $0$ and $2$, we are interested in paths in $\G(f)(b^0)$ starting from $0$ or $2$. But:
$$
\partial_0f(b^0) = f(b^0)+f(b^0+e^0)=f(b^0)+f(a^2)=a^3+a^3=0
$$
hence $0$ has no outgoing edge, and:
$$
\partial_2f(b^0) = f(b^0)+f(b^0+e^2)=a^3+b^0+e^2=e^0,
$$
because by Lemma \ref{lem:neighbors}, $b^0+e^2=a^0+e^{1,2}\notin A$ and is a fixed point, hence $2$ has only one outgoing edge $(2,0)$. Therefore $\G(f)(b^0)$ has no negative cycle.
\item[($c^0$)] 
Similarly, since $c^0$ has two degrees of freedom, $0$ and $1$, we are interested in paths in $\G(f)(c^0)$ starting from $0$ or $1$:
\begin{align*}
\partial_0f(c^0) &= f(c^0)+f(c^0+e^0)=f(c^0)+f(b^1)=a^3+a^4=e^3 \\
\partial_3f(c^0) &= f(c^0)+f(c^0+e^3)=f(c^0)+f(d^0)=a^3+a^4+e^1=e^{1,3} \\
\partial_1f(c^0) &= f(c^0)+f(c^0+e^1)=a^3+c^0+e^1=e^0
\end{align*}
because by Lemma \ref{lem:neighbors}, $c^0+e^1=a^0+e^{1,2}=b^0+e^2\notin A$ and is a fixed point, hence the only cycle though $0$ or $1$ in $\G(f)(c^0)$ passes though both $0$ and $1$, and is therefore positive by Lemma \ref{lem:sgnpar}.
\item[($d^0$)] 
Since $d^0$ has one degree of freedom, $0$, we need to compute:
$$
\partial_0f(d^0)=f(d^0)+f(d^0+e^0)=d^0+e^0+d^0+e^0=0
$$
because by Lemma \ref{lem:neighbors}, $d^0+e^0=a^0+e^{0,2,3}\notin A$ and is a fixed point. (The exceptional neighbors of $d^0$ in $A$ when $n=7$ satisfy $d^5=d^0+e^4$ and $d^{-5}=d^0+e^{-1}=e^6$, and are thus different from $d^0+e^0$.) Hence $0$ has no outgoing edge and $\G(f)(d^0)$ has no negative cycle.
\end{description}
This concludes the proof of Theorem B.

\section{Additionnal remarks}
\label{sec:addition}

\subsection{Non-expansive networks}
In the non-expansive case, it is actually possible to slightly improve the result in \cite{Ric11} (point $3$ of Theorem \ref{th:knownneg}) as follows. This improvement has been independently proved in \cite{Ric15} (Remark 4).
\begin{remark}
\label{th:nonexpneg}
Let $f:\B^n\rightarrow\B^n$ be a non-expansive Boolean network. If $f$ has a cyclic attractor, then $\G(f)$ has a local negative cycle.
\end{remark}
We recall the definition of a subcube. Given $x\in\B^n$ and $I\subseteq\{1,\ldots,n\}$, the subset $x[I]$ consists of all points $y$ such that $y_i=x_i$ for each $i\notin I$; subsets of the form $x[I]$ are called $I$-\emph{subcubes}, or simply \emph{subcubes} of $\B^n$. For any subcube $\kappa$, $f\mathpunct{\restriction}_\kappa:\kappa\rightarrow\kappa$ is defined in the obvious way: if $\kappa=x[I]$ and $y\in\kappa$,
$$
(f\mathpunct{\restriction}_\kappa(y))_i=
\begin{cases}
f_i(y) & \text{if $i\in I$} \\
x_i & \text{otherwise.}
\end{cases}
$$
This definition is compatible with interaction graphs: $\G(f\mathpunct{\restriction}_\kappa)(y)$ is the signed subgraph of $\G(f)(y)$ induced by $I$.
\begin{proof}
Assume that $f$ has a cyclic attractor. Let $\kappa$ be a minimal subcube such that $f\mathpunct{\restriction}_\kappa$ has a cyclic attractor, and let $I$ be such that $\kappa$ is an $I$-subcube. Since $f$ is non-expansive, $g=f\mathpunct{\restriction}_\kappa$ has to be non-expansive as well \cite{Ric11}.

Besides, for any $i\in I$, there exist $x,y\in\kappa$ such that $x_i\neq y_i$ and $g(x)+x=g(y)+y=e^i$. To see this, let $A$ be the set of points of a cyclic attractor of $g$, and let $\kappa^0$ (\resp $\kappa^1$) be the subcube of $\kappa$ defined by $x_i=0$ (\resp $x_i=1$). By minimality of $\kappa$, $f\mathpunct{\restriction}_{\kappa^0}$ and $f\mathpunct{\restriction}_{\kappa^1}$ have no cyclic attractor, hence, in particular, any point in $A\cap\kappa^0$ (\resp $A\cap\kappa^1$) is the beginning of a trajectory to some fixed point $x$ of $g\mathpunct{\restriction}_{\kappa^0}$ (\resp $y$ of $g\mathpunct{\restriction}_{\kappa^1}$). Since $x,y\in A$, they are not fixed points of $g$, thus $g(x)=x+e^i$ and $g(y)=y+e^i$, as required.

By \cite{Ric11}, the existence of such a pair $x,y$ suffices to entail a local negative cycle in $\G(g)$, hence in $\G(f)$, as the signed graphs $\G(g)(z)$ are induced subgraphs of $\G(f)(z)$.
\end{proof}

\subsection{Hoopings and invertible Jacobian matrices}
This is a minor remark on one of the few techniques proposed for proving the existence of a local negative cycle.

It is proved in \cite{Rue14} that if $x$ has odd out-degree in $\Gamma(f)$ and the Jacobian matrix $\J(f)(x)$ is invertible, then $\G(f)(x)$ has a negative cycle. Actually, if a \emph{hooping} is a spanning subgraph consisting of disjoint cycles \cite{Sou03}, whose \emph{sign} is the product of the sign of its edges (the definition of signs here differs from \cite{Sou03}), then a network with $\J(f)(x)$ invertible at some $x$ with odd out-degree must have an odd number of negative hoopings (by Lemma \ref{lem:sgnpar}), thus at least one, hence it has a negative cycle.

But for large enough $n$, if $\Gamma(f)$ consists of the antipodal attractive cycle $(0,\ldots,e^{1,\ldots,n-1},\overline{0},\ldots,\overline{e^{1,\ldots,n-1}},0)$ and all other points fixed, then the matrices of all local interaction graphs $\G(f)(x)$ have several equal columns (thus no hooping) and are consequently not invertible. Showing one of these two sufficient conditions for a local negative cycle (hooping or invertible Jacobian matrix) is therefore hopeless, at least for general Boolean networks.

\subsection{Reduction and local interaction graphs}
Reduction (Section \ref{sec:red}) does not preserve local interaction graphs in general, but it is possible to give a more precise account of the relationship.
\begin{proposition}
\label{prop:redjac}
If $\G(f)$ has no loop on $n$, $i,j\in\{1,\ldots,n-1\}$ and $x\in\B^{n-1}$, then $\partial_jf'_i(x)=\partial_jf_i(x')+\partial_jf_n(x')\cdot\partial_nf_i(x'+e^j)$.
\end{proposition}
\begin{proof}
\begin{align*}
\partial_jf'_i(x)+\partial_jf_i(x') = \ & f'_i(x)+f'_i(x+e^j)+f_i(x')+f_i(x'+e^j) \\
= \ & f_i(x,f_n(x,-))+f_i(x+e^j,f_n(x+e^j,-)) \\
& +f_i(x,f_n(x,-))+f_i(x+e^j,f_n(x,-)) \\
= \ & f_i(x+e^j,f_n(x+e^j,-))+f_i(x+e^j,f_n(x,-)).
\end{align*}
On the other hand:
$$
\partial_jf_n(x')=f_n(x')+f_n(x'+e^j)=f_n(x,-)+f_n(x+e^j,-).
$$
Therefore, if $\partial_jf_n(x')=0$, we have $f_n(x,-)=f_n(x+e^j,-)$ and $\partial_jf'_i(x)+\partial_jf_i(x')=0$. Otherwise, 
$\partial_jf_n(x')=1$ and $f_n(x,-)\neq f_n(x+e^j,-)$, whence
\begin{align*}
\partial_jf'_i(x)+\partial_jf_i(x') &= f_i(x+e^j,0)+f_i(x+e^j,1) \\
&= \partial_nf_i(x'+e^j).
\end{align*}
In both cases, $\partial_jf'_i(x)+\partial_jf_i(x')=\partial_jf_n(x')\cdot\partial_nf_i(x'+e^j)$, as required.
\end{proof}
This has an immediate consequence for global interaction graphs: an edge $(j,i)$ of $\G(f')$ is either an edge $(j,i)$ of $\G(f)$ or the result of a pair of consecutive edges $(j,n)$ and $(n,i)$ in $\G(f)$. In Section \ref{sec:qdf}, we shall say that the edges $(j,i),(j,n),(n,i)$ in $\G(f)$ are \emph{above} the edge $(j,i)$ of $\G(f')$, and more generally that a path or a cycle of $\G(f)$ is \emph{above} a path or a cycle of $\G(f')$ when it consists of edges above those of $\G(f')$.

When moreover $f$ and $n$ are such that for any $i,j$, either $(j,n)$ or $(j,i)$ is not an edge of the global graph $\G(f)$, we may then note that $\partial_jf'_i(x)$ equals
\begin{align*}
\text{either } & \partial_jf_i(x') \\
\text{or } & \partial_jf_n(x')\cdot\partial_nf_i(x'+e^j)=\partial_jf_n(x')\cdot\partial_nf_i(x').
\end{align*}
The last equality holds because:
$$
\partial_nf_i(x'+e^j)+\partial_nf_i(x')=\partial_jf_i(x'+e^n)+\partial_jf_i(x')=0.
$$
In that case\footnote{Using second order derivatives as in \cite{Rue14}, we may observe that the above line equals simply $\partial_{n,j}f_i(x')$.}, more can be said about local interaction graphs: an edge $(j,i)$ in $\G(f')(x)$ is then the result of either an edge $(j,i)$ in $\G(f)(x')$ (first case in the above alternative) or a pair of consecutive edges $(j,n)$ and $(n,i)$ in $\G(f)(x')$ (second case).

\section{Conclusion}
\label{sec:concl}

Theorems A and B set limits to the possible relationships between local negative cycles and asymptotic dynamical properties. There is a rather straightforward but interesting abstract difference between these two kinds of properties: the structural properties (absence of some local cycles) are hereditary (in the sense that if a network has the property, then so has the \emph{subnetwork} induced by any subcube), while the dynamic properties are clearly not. This remark is by the way very useful in the case of positive cycles \cite{Rue14}. Going further in this direction for general Boolean networks probably requires to find either hereditary counterparts on the dynamics side or a non-hereditary property involving cycles.

On the other hand, we mentioned in Section \ref{sec:dfcount} a open question for the class of and-nets: do and-nets with an attractive cycle (a non necessarily antipodal one) have a local negative cycle? A positive answer would improve point $5$ of Theorem \ref{th:knownneg}. Understanding the role of local negative cycles in classes of networks generalizing and-nets, like canalyzing or nested canalyzing networks \cite{Lau07}, seems interesting.

\appendix
\section{Characterization of isometries of $\B^n$}
\label{sec:isomproof}

The claim about isometries at the beginning of Section \ref{sec:def} amounts to verify the following:
\begin{proposition}
\label{lem:isom}
If $U$ is an isometry of $\B^n$, there exists a (unique) permutation $\sigma\in\Sig_n$ such that for any $x\in\B^n$ and $I\subseteq\{1,\ldots,n\}$:
$$
U(x+e^I)=U(x)+e^{\sigma(I)},
$$
where $\sigma(I)=\{\sigma(i),i\in I\}$.
\end{proposition}
\begin{proof}
The claim holds trivially for $I=\varnothing$.

For any $i\in\{1,\ldots,n\}$, $d(U(x+e^i),U(x))=1$ since $U$ is an isometry, hence $U(x+e^i)=U(x)+e^{i'}$ for some $i'$. This $i'$ does not depend on $x$ because unit squares in $\B^n$ such as:
$$
\begin{array}{@{\extracolsep{6mm}}cccc}
&&& \\
\rnode{x}{x} & \rnode{xx}{x+e^{k_1}} & \rnode{xxx}{x+e^{k_1,k_2}} & \\[4mm]
&&& \cdots \\[4mm]
\rnode{y}{x+e^i} & \rnode{yy}{x+e^{i,k_1}} & \rnode{yyy}{x+e^{i,k_1,k_2}} & \\
\psset{nodesep=3pt,arrows=-}
\ncline{x}{xx}\naput{k_1}\ncline{xx}{xxx}\naput{k_2}
\ncline{y}{yy}\nbput{k_1}\ncline{yy}{yyy}\nbput{k_2}
\ncline{x}{y}\nbput{i}\ncline{xx}{yy}\nbput{i}\ncline{xxx}{yyy}\naput{i}
\end{array}
$$
must be mapped by the isometry $U$ to unit squares:
$$
\begin{array}{@{\extracolsep{6mm}}cccc}
&&& \\
\rnode{ux}{U(x)} & \rnode{uxx}{U(x+e^{k_1})} & \rnode{uxxx}{U(x+e^{k_1,k_2})} & \\[4mm]
&&& \cdots \\[4mm]
\rnode{uy}{U(x+e^i)} & \rnode{uyy}{U(x+e^{i,k_1})} & \rnode{uyyy}{U(x+e^{i,k_1,k_2})} & \\
\psset{nodesep=3pt,arrows=-}
\ncline{ux}{uxx}\naput{k'_1}\ncline{uxx}{uxxx}\naput{k'_2}
\ncline{uy}{uyy}\nbput{k'_1}\ncline{uyy}{uyyy}\nbput{k'_2}
\ncline{ux}{uy}\nbput{i'}\ncline{uxx}{uyy}\nbput{i'}\ncline{uxxx}{uyyy}\naput{i'}
\end{array}
$$
so that for any $y$, letting $e^K=x+y$, we have $U(y+e^i)=U(x+e^K+e^i)=U(x+e^K)+e^{i'}$.
Hence we may define $\sigma(i)=i'$. The function $\sigma$ thus defined is a permutation because it is injective: if $\sigma(i)=\sigma(j)$, then $U(e^i)=U(0)+e^{\sigma(i)}=U(e^j)$ and $i=j$.

We now proceed by induction on the cardinality of $I\neq\varnothing$. Let $i\in I$:
\begin{align*}
U(x+e^I) &= U(x+e^i+e^{I\setminus\{i\}}) \\
&= U(x+e^i)+e^{\sigma(I\setminus\{i\})} \quad \text{by induction} \\
&= U(x)+e^{\sigma(i)}+e^{\sigma(I\setminus\{i\})} \\
&= U(x)+e^{\sigma(I)}
\end{align*}
concludes the proof.
\end{proof}

\section{Proof of Proposition \ref{prop:wd} and Lemma \ref{lem:neighbors}}
\label{sec:appen}

Let $\inter=\{i,i+1,\ldots,j-1\}$, where integers are considered modulo $n$, and $\bigtriangleup$ denote the symmetric difference of sets.

We observe that knowing the $b^i$'s which are neighbors of $a^0$ gives, by symmetry, the $a^i$'s neighbors of $b^0$, etc, so that only $4+3+2+1$ cases need to be considered. The following is a straighforward computation.
\begin{lemma}
\label{lem:abcd}
Let $-n<i\leqslant n$ and $S_i=\zinter$ if $i\geqslant 0$, $\interz$ if $i<0$. The sets $S_i$ are therefore a family of $2n$ sets of integers modulo $n$, and $S_i$ has cardinality $|i|$.
\begin{alignat*}{3}
a^0+a^i &= e^{S_i} & b^0+c^i &= e^{S_i\bigtriangleup\{1\}\bigtriangleup\{i+2\}} \\
a^0+b^i &= e^{S_i\bigtriangleup\{i+1\}} & b^0+d^i &= e^{S_i\bigtriangleup\{1\}\bigtriangleup\{i+2,i+3\}} \\
a^0+c^i &= e^{S_i\bigtriangleup\{i+2\}} & c^0+c^i &= e^{S_i\bigtriangleup\{2\}\bigtriangleup\{i+2\}} \\
a^0+d^i &= e^{S_i\bigtriangleup\{i+2,i+3\}} & \qquad\qquad c^0+d^i &= e^{S_i\bigtriangleup\{2\}\bigtriangleup\{i+2,i+3\}} \\
b^0+b^i &= e^{S_i\bigtriangleup\{1\}\bigtriangleup\{i+1\}} & d^0+d^i &= e^{S_i\bigtriangleup\{2,3\}\bigtriangleup\{i+2,i+3\}}
\end{alignat*}
\end{lemma}
In the above equations, we write, e.g., $\{1\}\bigtriangleup\{i+1\}$ instead of $\{1,i+1\}$, because $\{1\}\bigtriangleup\{i+1\}=\varnothing$ for $i=0$ or $n$, while the notation on the right suggests a singleton in this case. On the other hand, since $n\geqslant 7$, $i+2\neq i+3 \text{ mod } n$ for instance, so there is no need for a $\bigtriangleup$ notation here.

To prove Proposition \ref{prop:wd} and Lemma \ref{lem:neighbors}, we have to consider the 10 pairs of points and tell for which $i$ they can be equal or neighbors.
\begin{description}
\item[($a^0$ and $a^i$)] 
They are equal when $S_i$ is empty, \ie $i=0$, and neighbors when $S_i$ is empty or a singleton, \ie $-1\leqslant i\leqslant 1$, whence the three neighbors $a^{-1},a^0,a^1$ of $a^0$.
\item[($a^0$ and $b^i$)] 
They are equal when $S_i\bigtriangleup\{i+1\}$ is empty, which is possible only if $S_i$ is a singleton, \ie $i=1$ and $S_i=\{0\}$, or $i=-1$ and $S_i=\{-1\}$. In the first case, $i+1=2\neq 0 \text{ mod } n$, and in the second case, $i+1=0\neq -1 \text{ mod } n$. Thus $i+1\notin S_i$, and $a^0$ and $b^i$ are never equal. They are neighbors when $S_i\bigtriangleup\{i+1\}$ is a singleton, which is possible only if $S_i$ has cardinality $0$ or $2$. In the first case, $i=0$, and this corresponds to the neighbor $b^0$ of $a^0$. In the second case:
$$
S_i\bigtriangleup\{i+1\}=
\begin{cases}
\{0,1\}\bigtriangleup\{3\}=\{0,1,3\} & \text{ if $i=2$} \\
\{-2,-1\}\bigtriangleup\{-1\}=\{-2\} & \text{ if $i=-2$},
\end{cases}
$$
whence the neighbor $b^{-2}$ of $a^0$.
\item[($a^0$ and $c^i$)] 
They are equal when $S_i\bigtriangleup\{i+2\}$ is empty, which is possible only if $i=1$ or $-1$. In the first case, $S_i=\{0\}$ and $i+2=3\neq 0 \text{ mod } n$, and in the second case, $S_i=\{-1\}$ and $i+2=1\neq -1 \text{ mod } n$. Thus $a^0$ and $c^i$ are never equal. They are neighbors when $S_i\bigtriangleup\{i+2\}$ is a singleton, which is possible only if $S_i$ has cardinality $0$ or $2$. In the first case, $i=0$, and this corresponds to the neighbor $c^0$ of $a^0$. In the second case:
$$
S_i\bigtriangleup\{i+2\}=
\begin{cases}
\{0,1\}\bigtriangleup\{4\}=\{0,1,4\} & \text{ if $i=2$} \\
\{-2,-1\}\bigtriangleup\{0\}=\{-2,-1,0\} & \text{ if $i=-2$},
\end{cases}
$$
and $S_i\bigtriangleup\{i+2\}$ is not a singleton.
\item[($a^0$ and $d^i$)] 
They are equal when $S_i\bigtriangleup\{i+2,i+3\}$ is empty, which is possible only if $S_i$ has cardinality $2$, \ie $i=2$ or $-2$, and then $i+2\notin S_i$ as above. Thus $a^0$ and $d^i$ are never equal. They are neighbors when $S_i\bigtriangleup\{i+2,i+3\}$ is a singleton, \ie when $S_i$ has cardinality $1$ or $3$. There are 4 cases:
$$
S_i\bigtriangleup\{i+2,i+3\}=
\begin{cases}
\{0\}\bigtriangleup\{3,4\}=\{0,3,4\} & \text{ if $i=1$} \\
\{-1\}\bigtriangleup\{1,2\}=\{-1,1,2\} & \text{ if $i=-1$} \\
\{0,1,2\}\bigtriangleup\{5,6\}=\{0,1,2,5,6\} & \text{ if $i=3$} \\
\{-3,-2,-1\}\bigtriangleup\{-1,0\}=\{-3,-2,0\} & \text{ if $i=-3$,} \\
\end{cases}
$$
under the assumption $n\geqslant 7$. In any case, $S_i\bigtriangleup\{i+2,i+3\}$ is not a singleton, and $a^0$ and $d^i$ are not neighbors.
\item[($b^0$ and $b^i$)] 
For $i\neq 0$, they are equal or neighbors when $S_i\bigtriangleup\{1\}\bigtriangleup\{i+1\}$ is empty or a singleton, which is possible only if $S_i$ has cardinality at most $3$. There are 6 cases:
$$
S_i\bigtriangleup\{1\}\bigtriangleup\{i+1\}=
\begin{cases}
\{0\}\bigtriangleup\{1\}\bigtriangleup\{2\}=\{0,1,2\} & \text{ if $i=1$} \\
\{-1\}\bigtriangleup\{1\}\bigtriangleup\{0\}=\{-1,0,1\} & \text{ if $i=-1$} \\
\{0,1\}\bigtriangleup\{1\}\bigtriangleup\{3\}=\{0,3\} & \text{ if $i=2$} \\
\{-2,-1\}\bigtriangleup\{1\}\bigtriangleup\{-1\}=\{-2,1\} & \text{ if $i=-2$} \\
\{0,1,2\}\bigtriangleup\{1\}\bigtriangleup\{4\}=\{0,2,4\} & \text{ if $i=3$} \\
\{-3,-2,-1\}\bigtriangleup\{1\}\bigtriangleup\{-2\} & \text{ if $i=-3$,} \\
\end{cases}
$$
under the assumption $n\geqslant 7$, and in any case, $S_i\bigtriangleup\{1\}\bigtriangleup\{i+1\}$ is neither empty nor a singleton.
\item[($b^0$ and $c^i$)] 
They are equal or neighbors when $S_i\bigtriangleup\{1\}\bigtriangleup\{i+2\}$ is empty or a singleton, which is possible only if $S_i$ has cardinality at most $3$. There are 7 cases:
$$
S_i\bigtriangleup\{1\}\bigtriangleup\{i+2\}=
\begin{cases}
\{1,2\} & \text{ if $i=0$} \\
\{0,1,3\} & \text{ if $i=1$} \\
\{-1\} & \text{ if $i=-1$} \\
\{0,1\}\bigtriangleup\{1\}\bigtriangleup\{4\}=\{0,4\} & \text{ if $i=2$} \\
\{-2,-1\}\bigtriangleup\{1\}\bigtriangleup\{0\}=\{-2,-1,0,1\} & \text{ if $i=-2$} \\
\{0,1,2\}\bigtriangleup\{1\}\bigtriangleup\{5\}=\{0,2,5\} & \text{ if $i=3$} \\
\{-3,-2,-1\}\bigtriangleup\{1\}\bigtriangleup\{-1\} & \text{ if $i=-3$},
\end{cases}
$$
under the assumption $n\geqslant 7$. Thus $b^0$ and $c^i$ are never equal, and neighbors when $i=-1$, which corresponds to the neighbor $c^{-1}$ of $b^0$.
\item[($b^0$ and $d^i$)] 
They are equal or neighbors when $S_i\bigtriangleup\{1\}\bigtriangleup\{i+2,i+3\}$ is empty or a singleton, which is possible only if $S_i$ has cardinality at most $4$. There are 9 cases: $S_i\bigtriangleup\{1\}\bigtriangleup\{i+2,i+3\}$ equals
$$
\begin{cases}
\{1,2,3\} & \text{ if $i=0$} \\
\{0,1,3,4\} & \text{ if $i=1$} \\
\{-1,2\} & \text{ if $i=-1$} \\
\{0,4,5\} & \text{ if $i=2$} \\
\{-2,-1\}\bigtriangleup\{1\}\bigtriangleup\{0,1\}=\{-2,-1,0\} & \text{ if $i=-2$} \\
\{0,1,2\}\bigtriangleup\{1\}\bigtriangleup\{5,6\}=\{0,2,5,6\} & \text{ if $i=3$} \\
\{-3,-2,-1\}\bigtriangleup\{1\}\bigtriangleup\{-1,0\}=\{-3,-2,0,1\} & \text{ if $i=-3$} \\
\{0,1,2,3\}\bigtriangleup\{1\}\bigtriangleup\{6,7\}=\{0,2,3\}\bigtriangleup\{6,7\} & \text{ if $i=4$} \\
\{-4,-3,-2,-1\}\bigtriangleup\{1\}\bigtriangleup\{-2-1\}=\{-4,-3,1\} & \text{ if $i=-4$},
\end{cases}
$$
under the assumption $n\geqslant 7$. The case $i=4$ is a $3$-element set if $n=7$. Thus $b^0$ and $d^i$ are never equal or neighbors.
\item[($c^0$ and $c^i$)] 
If $i\neq 0$, they are equal or neighbors when $S_i\bigtriangleup\{2\}\bigtriangleup\{i+2\}$ is empty or a singleton, which is possible only if $S_i$ has cardinality at most $3$. There are 6 cases: $S_i\bigtriangleup\{2\}\bigtriangleup\{i+2\}$ equals
$$
\begin{cases}
\{0,2,3\} & \text{ if $i=1$} \\
\{-1,2,1\} & \text{ if $i=-1$} \\
\{0,1,2,4\} & \text{ if $i=2$} \\
\{-2,-1\}\bigtriangleup\{2\}\bigtriangleup\{0\}=\{-2,-1,0,2\} & \text{ if $i=-2$} \\
\{0,1,2\}\bigtriangleup\{2\}\bigtriangleup\{5\}=\{0,1,5\} & \text{ if $i=3$} \\
\{-3,-2,-1\}\bigtriangleup\{2\}\bigtriangleup\{-1\}=\{-3,-2,2\} & \text{ if $i=-3$},
\end{cases}
$$
under the assumption $n\geqslant 7$. Thus $c^0$ and $c^i$ are never equal or neighbors for $i\neq 0$.
\item[($c^0$ and $d^i$)] 
They are equal or neighbors when $S_i\bigtriangleup\{2\}\bigtriangleup\{i+2,i+3\}$ is empty or a singleton, which is possible only if $S_i$ has cardinality at most $4$. There are 9 cases: $S_i\bigtriangleup\{2\}\bigtriangleup\{i+2,i+3\}$ equals
$$
\begin{cases}
\{3\} & \text{ if $i=0$} \\
\{0,2,3,4\} & \text{ if $i=1$} \\
\{-1,1\} & \text{ if $i=-1$} \\
\{0,1,2,4,5\} & \text{ if $i=2$} \\
\{-2,-1\}\bigtriangleup\{2\}\bigtriangleup\{0,1\}=\{-2,-1,0,1,2\} & \text{ if $i=-2$} \\
\{0,1,2\}\bigtriangleup\{2\}\bigtriangleup\{5,6\}=\{0,1,5,6\} & \text{ if $i=3$} \\
\{-3,-2,-1\}\bigtriangleup\{2\}\bigtriangleup\{-1,0\}=\{-3,-2,0,2\} & \text{ if $i=-3$} \\
\{0,1,2,3\}\bigtriangleup\{2\}\bigtriangleup\{6,7\}=\{0,1,3\}\bigtriangleup\{6,7\} & \text{ if $i=4$} \\
\{-4,-3,-2,-1\}\bigtriangleup\{2\}\bigtriangleup\{-2-1\}=\{-4,-3,2\} & \text{ if $i=-4$},
\end{cases}
$$
under the assumption $n\geqslant 7$. The case $i=0$ gives the neighbor $d^0$ of $c^0$. Note that the case $i=4$ would be a singleton if $n=6$, so for the first time we use the full hypothesis $n\geqslant 7$.
\item[($d^0$ and $d^i$)] 
If $i\neq 0$, they are equal or neighbors when $S_i\bigtriangleup\{2,3\}\bigtriangleup\{i+2,i+3\}$ is empty or a singleton, which is possible only if $S_i$ has cardinality at most $5$. There are 10 cases: $S_i\bigtriangleup\{2,3\}\bigtriangleup\{i+2,i+3\}$ equals
$$
\begin{cases}
\{0,2,4\} & \text{ if $i=1$} \\
\{-1,1,3\} & \text{ if $i=-1$} \\
\{0,1,2,3,4,5\} & \text{ if $i=2$} \\
\{-2,-1\}\bigtriangleup\{2,3\}\bigtriangleup\{0,1\}=\{-2,-1,0,1,2,3\} & \text{ if $i=-2$} \\
\{0,1,2\}\bigtriangleup\{2,3\}\bigtriangleup\{5,6\}=\{0,1,3\}\bigtriangleup\{5,6\} & \text{ if $i=3$} \\
\{-3,-2,-1\}\bigtriangleup\{2,3\}\bigtriangleup\{-1,0\} & \text{ if $i=-3$} \\
\{0,1,2,3\}\bigtriangleup\{2,3\}\bigtriangleup\{6,7\}=\{0,1\}\bigtriangleup\{6,7\} & \text{ if $i=4$} \\
\intervalle{-4,0}\bigtriangleup\{2,3\}\bigtriangleup\{-2-1\}=\{-4,-3\}\bigtriangleup\{2,3\} & \text{ if $i=-4$} \\
\{0,1,2,3,4\}\bigtriangleup\{2,3\}\bigtriangleup\{7,8\}=\{0,1,4\}\bigtriangleup\{7,8\} & \text{ if $i=5$} \\%
\intervalle{-5,0}\bigtriangleup\{2,3\}\bigtriangleup\{-3,-2\}=\{-5,-4,-1\}\bigtriangleup\{2,3\} & \text{ if $i=-5$},%
\end{cases}
$$
under the assumption $n\geqslant 7$. The cases $i=5$ and $i=-5$ give, when $n=7$, the neighbors $d^5$ and $d^{-5}$ of $d^0$, with $d^0+d^5=e^4$ and $d^0+d^{-5}=e^{-1}=e^6$. The cases $i=4$ and $i=-4$ would be empty if $n=6$, so again here we use the full hypothesis $n\geqslant 7$.
\end{description}


\bibliographystyle{plain}

\end{document}